\documentclass[a4paper,10pt]{amsart}

\usepackage{amssymb,amsthm,amsmath,verbatim}
\usepackage{t1enc}
\usepackage{enumerate}
\usepackage{color}
\usepackage{tikz}
\usepackage{breqn}
\usepackage{xmpmulti}
\usepackage{psfrag}

\usepackage{todonotes}

 \usepackage[usenames,dvipsnames]{pstricks}
 \usepackage{epsfig}
\usepackage{pst-grad}
\usepackage{pstricks,pst-node}
\usepackage{float}

\newenvironment{tproof}{
  
  \begin{proof}
}{\end{proof}}

\newenvironment{cproof}{
  
  \begin{proof}
}{\end{proof}}

\usepackage[left=3.5 cm, right=3.5cm]{geometry}
\numberwithin{equation}{section}

\usepackage{hyperref}
\newcommand{\force}{{\hspace{0.02 cm}\Vdash}}

\newtheorem{prop}{Proposition}[section]
\newtheorem{dfn}[prop]{Definition}
\newtheorem{lemma}[prop]{Lemma}
\newtheorem{fact}[prop]{Fact}
\newtheorem{cor}[prop]{Corollary}
\newtheorem{thm}[prop]{Theorem}

\newtheorem{clm}[prop]{Claim}
\newtheorem{obs}[prop]{Observation}

\newtheorem{prob}[prop]{Problem}

\newcommand{\sqt}{\sigma \mb Q}
\newcommand{\ssqt}{\sigma' \mb Q}
\newcommand{\mc}[1]{\mathcal{#1}}
\newcommand{\mb}[1]{\mathbb{#1}}

\newcommand{\oo}{\omega}
\newcommand{\uhr}{\upharpoonright}
\newcommand{\omg}{{\omega_1}}

\newcommand{\ecl}[1]{\langle #1 \rangle}

\DeclareMathOperator{\htt}{ht}

\def\<{\left\langle}
\def\>{\right\rangle}
\def\br#1;#2;{\bigl[ {#1} \bigr]^ {#2} }

\newcommand{\mf}[1]{\mathfrak{#1}}

\newcommand{\unif}[3]{\textmd{Unif}_{#1}(#2,#3)}
\newcommand{\unifc}[3]{\textmd{cUnif}_{#1}(#2,#3)}

\newcommand{\setm}{\setminus}

\newcommand{\subs}{\subset}
\newcommand{\dom}{\operatorname{dom}}
\newcommand{\ran}{\operatorname{ran}}

\newcommand{\vareps}{\varepsilon}
\newcommand{\smf}{\hspace{0.008 cm}^\smallfrown}

\newcommand{\seal}{\mf{sl}}

\title[Ladder system uniformization on trees I \& II]{Ladder system uniformization on trees I \& II}

\date{\today}

  \author{D\'aniel T. Soukup}
\address[D. T. Soukup]{Universität Wien, 
Kurt Gödel Research Center for Mathematical Logic, 
Währinger Strasse 25,
1090 Wien, Austria}
 \email[Corresponding author]{daniel.soukup@univie.ac.at}

 \urladdr{http://www.logic.univie.ac.at/$\sim  $soukupd73/}

\makeatletter
\newtheorem*{rep@theorem}{\rep@title}
\newcommand{\newreptheorem}[2]{%
\newenvironment{rep#1}[1]{%
 \def\rep@title{#2 \ref{##1}}%
 \begin{rep@theorem}}%
 {\end{rep@theorem}}}
\makeatother

\newreptheorem{corollary}{Corollary}
\newreptheorem{theorem}{Theorem}

\subjclass[2010]{03E05,03E35,03E50}
\keywords{ladder system, uniformization, Suslin tree, special tree, Aronszajn tree, diamond, colouring}

\begin{document}
 \begin{abstract}  

Suppose that $T$ is a tree of height $\omg$. We say that a ladder system colouring $(f_\alpha)_{\alpha\in \lim\omg}$ has a \emph{$T$-uniformization} if there is  a function $\varphi$ defined on a subtree $S$ of $T$ so that for any $s\in S_\alpha$ of limit height and almost all $\xi\in \dom f_\alpha$, $\varphi(s\uhr \xi)=f_\alpha(\xi)$. In sharp contrast to the classical theory of uniformizations on $\omg$, J. Moore proved that CH is consistent with the statement that any ladder system colouring has a $T$-uniformization (for any Aronszajn tree $T$). 
However, we show that if $S$ is a Suslin tree then (i) CH implies that there is a ladder system colouring without $S$-uniformization; (ii) the restricted forcing axiom $MA(S)$ implies that any ladder system colouring has an $\omg$-uniformization. For an arbitrary Aronszajn tree $T$, we show how diamond-type assumptions affect the existence of ladder system colourings without a $T$-uniformization. 

Furthermore, it is consistent that for any Aronszajn tree $T$ and ladder system $\mathbf C$  there is a colouring of $\mathbf C$ without a $T$-uniformization; however, and quite surprisingly,  $\diamondsuit^+$ implies that for any ladder system $\mathbf C$ there is an Aronszajn tree $T$ so that any monochromatic colouring of $\mathbf C$ has a $T$-uniformization. We also prove positive uniformization results  in ZFC  for some well-studied trees of size continuum.

\end{abstract}
\maketitle

\setcounter{tocdepth}{1}
\tableofcontents

\section{Introduction}

A sequence $\textbf{C}=(C_\alpha)_{\alpha\in \lim \omg}$ indexed by the set of countable limit ordinals is a \emph{ladder system on $\omg$} if $C_\alpha$ is a cofinal subset of $\alpha$ of  type $\omega$. A colouring of the ladder system $\textbf C$ is a sequence of maps $\textbf{f}=(f_\alpha)_{\alpha\in \lim\omg}$ so that  $\dom f_\alpha=C_\alpha$. Now, given such a ladder system colouring, one might ask if there is a single global function $\varphi$ defined on $\omg$ that agrees with all the local colourings $f_\alpha$ almost everywhere on $C_\alpha$; we refer to such a global map $\varphi$ as an \emph{$\omg$-uniformization}. 

 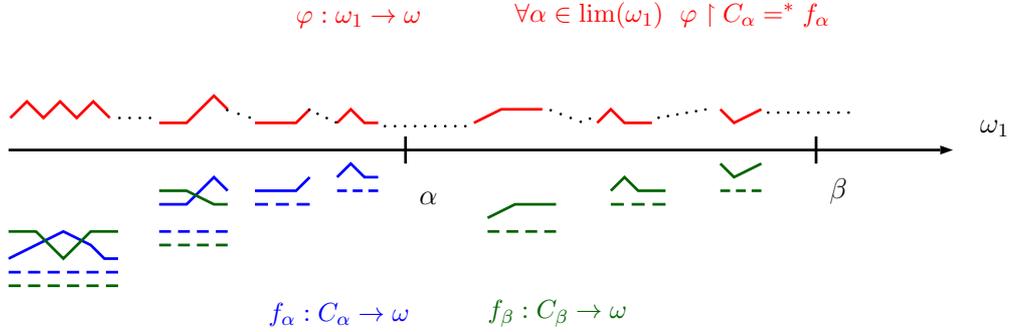
\begin{figure}[H] 
 \centering

\psscalebox{0.9 0.9} 
{
\begin{pspicture}(0,-2.3778028)(17.788937,2.3778028)
\definecolor{colour0}{rgb}{0.0,0.39215687,0.0}
\psline[linecolor=black, linewidth=0.04, arrowsize=0.05291667cm 2.0,arrowlength=1.4,arrowinset=0.0]{->}(0.0089373775,0.29813477)(13.808937,0.29813477)
\rput[bl](14.208938,0.49813476){\Large{$\omega_1$}}
\psline[linecolor=black, linewidth=0.04](5.8089375,0.49813476)(5.8089375,0.09813476)
\psline[linecolor=black, linewidth=0.04](11.808937,0.49813476)(11.808937,0.09813476)
\psline[linecolor=blue, linewidth=0.04, linestyle=dashed, dash=0.17638889cm 0.10583334cm](0.0089373775,-1.5018653)(1.6089374,-1.5018653)
\psline[linecolor=blue, linewidth=0.04, linestyle=dashed, dash=0.17638889cm 0.10583334cm](2.2089374,-0.90186524)(3.2089374,-0.90186524)
\psline[linecolor=blue, linewidth=0.04, linestyle=dashed, dash=0.17638889cm 0.10583334cm](3.6089373,-0.5018652)(4.4089375,-0.5018652)
\psline[linecolor=blue, linewidth=0.04, linestyle=dashed, dash=0.17638889cm 0.10583334cm](4.8089375,-0.30186522)(5.4089375,-0.30186522)
\psline[linecolor=colour0, linewidth=0.04, linestyle=dashed, dash=0.17638889cm 0.10583334cm](8.808937,-0.5018652)(9.608937,-0.5018652)
\psline[linecolor=colour0, linewidth=0.04, linestyle=dashed, dash=0.17638889cm 0.10583334cm](10.408937,-0.30186522)(11.008938,-0.30186522)
\psline[linecolor=colour0, linewidth=0.04, linestyle=dashed, dash=0.17638889cm 0.10583334cm](7.0089374,-0.90186524)(8.008938,-0.90186524)
\psline[linecolor=colour0, linewidth=0.04, linestyle=dashed, dash=0.17638889cm 0.10583334cm](2.2089374,-1.1018653)(3.2089374,-1.1018653)
\psline[linecolor=colour0, linewidth=0.04, linestyle=dashed, dash=0.17638889cm 0.10583334cm](0.0089373775,-1.7018652)(1.6089374,-1.7018652)
\rput[bl](6.0089374,-0.5018652){\Large{$\alpha$}}
\rput[bl](12.008938,-0.5018652){\Large{$\beta$}}
\rput[bl](7.0089374,-2.3018653){\textcolor{colour0}{\large{$f_\beta:C_\beta\to \omega$}}}
\rput[bl](3.8089373,-2.3018653){\textcolor{blue}{\large{$f_\alpha:C_\alpha\to \omega$}}}
\psline[linecolor=blue, linewidth=0.04](2.2089374,-0.5018652)(2.6089373,-0.5018652)(3.0089374,-0.10186523)(3.2089374,-0.30186522)
\psline[linecolor=colour0, linewidth=0.04](2.2089374,-0.30186522)(2.6089373,-0.30186522)(3.0089374,-0.5018652)(3.2089374,-0.5018652)
\psline[linecolor=blue, linewidth=0.04](0.0089373775,-1.3018652)(0.8089374,-0.90186524)(1.2089374,-1.1018653)(1.4089373,-1.3018652)(1.6089374,-1.3018652)(1.6089374,-1.3018652)
\psline[linecolor=colour0, linewidth=0.04](0.0089373775,-0.90186524)(0.40893736,-0.90186524)(0.8089374,-1.3018652)(1.2089374,-0.90186524)(1.6089374,-0.90186524)
\psline[linecolor=blue, linewidth=0.04](3.6089373,-0.30186522)(4.208937,-0.30186522)(4.4089375,-0.10186523)
\psline[linecolor=blue, linewidth=0.04](4.8089375,-0.10186523)(5.0089374,0.09813476)(5.208937,-0.10186523)(5.4089375,-0.10186523)
\psline[linecolor=colour0, linewidth=0.04](7.0089374,-0.70186526)(7.4089375,-0.5018652)(8.008938,-0.5018652)
\psline[linecolor=colour0, linewidth=0.04](8.808937,-0.30186522)(9.008938,-0.10186523)(9.208938,-0.30186522)(9.608937,-0.30186522)
\psline[linecolor=colour0, linewidth=0.04](10.408937,0.09813476)(10.608937,-0.10186523)(11.008938,0.09813476)
\psline[linecolor=red, linewidth=0.04](2.2089374,0.6981348)(2.6089373,0.6981348)(3.0089374,1.0981348)(3.2089374,0.89813477)
\psline[linecolor=red, linewidth=0.04](3.6089373,0.6981348)(4.208937,0.6981348)(4.4089375,0.89813477)
\psline[linecolor=red, linewidth=0.04](4.8089375,0.6981348)(5.0089374,0.89813477)(5.208937,0.6981348)(5.4089375,0.6981348)
\psline[linecolor=red, linewidth=0.04](6.8089375,0.6981348)(7.208937,0.89813477)(7.8089375,0.89813477)
\psline[linecolor=red, linewidth=0.04](8.608937,0.6981348)(8.808937,0.89813477)(9.008938,0.6981348)(9.408937,0.6981348)
\psline[linecolor=red, linewidth=0.04](10.408937,0.89813477)(10.608937,0.6981348)(11.008938,0.89813477)
\psline[linecolor=black, linewidth=0.04, linestyle=dotted, dotsep=0.10583334cm](1.6089374,0.77086204)(2.0937858,0.77086204)
\psline[linecolor=black, linewidth=0.04, linestyle=dotted, dotsep=0.10583334cm](3.184695,0.89207417)(3.5483313,0.77086204)
\psline[linecolor=black, linewidth=0.04, linestyle=dotted, dotsep=0.10583334cm](4.4237523,0.89813477)(4.3968163,0.89207417)(4.8089375,0.6981348)
\psline[linecolor=black, linewidth=0.04, linestyle=dotted, dotsep=0.10583334cm](5.4877253,0.6496499)(6.6998463,0.6496499)
\psline[linecolor=black, linewidth=0.04, linestyle=dotted, dotsep=0.10583334cm](7.9119678,0.89207417)(8.408937,0.6981348)(8.518028,0.77086204)
\psline[linecolor=black, linewidth=0.04, linestyle=dotted, dotsep=0.10583334cm](9.487725,0.77086204)(10.208938,0.89813477)
\psline[linecolor=red, linewidth=0.04](0.0331798,0.77086204)(0.27560404,1.0132862)(0.51802826,0.77086204)(0.7604525,1.0132862)(1.0028768,0.77086204)(1.245301,1.0132862)(1.4877253,0.77086204)
\rput[bl](4.208937,2.0981348){\textcolor{red}{\large{$\varphi:\omega_1\to \omega$}}}
\rput[bl](7.4089375,2.0981348){\textcolor{red}{\large{$\forall \alpha\in \lim(\omg)$\;  $\varphi\uhr C_\alpha=^*f_\alpha$}}}
\psline[linecolor=black, linewidth=0.04, linestyle=dotted, dotsep=0.10583334cm](11.087726,0.8496499)(12.299847,0.8496499)
\end{pspicture}
}
    \caption{Uniformization on $\omg$}
  \label{diag:unif0}
   \end{figure}

Apriori, nothing prevents the existence of $\varphi$ since two local maps $f_\alpha$ and $f_\beta$ are only defined on finitely many common points and finite errors are allowed (see Figure \ref{diag:unif0}).

The existence of $\omg$-uniformizations for arbitrary colourings is not decided by the usual ZFC axioms. This topic has been extensively studied due to various connections to algebra, in particular to the Whitehead problem and its relatives \cite{eklof1992uniformization,eklof1994hereditarily,whitehead1,whitehead2}, to topology \cite{balogh2004uniformization,devlin1979note,shelah1989consistent, watson1986locally}, and to fundamental questions in set theory \cite{larson2016automorphisms,justinmin}, in particular, to the study of forcing axioms that are compatible with the Continuum Hypothesis \cite{avraham1978consistency,shelah2017proper}. 

Our current interest lies in understanding a relatively new version of the uniformization property introduced by Justin Moore, a notion that played a key role in understanding uncountable minimal linear orders \cite{justinmin}. If $T$ is a tree of height $\omg$, we say that $S\subs T$ is a subtree if $S$ is downward closed and \emph{pruned in $T$} i.e., if any $s\in S$ has extensions with arbitrary large height below $\htt(T)=\omg$. Given some $s\in S$ and $\xi<\htt(s)$, we let $s\uhr \xi$ be the unique predecessor of $s$ in $S$ of height $\xi$.

Now, the main definition is the following.

\begin{dfn}\cite{justinmin}\label{dfn:Tunif}
Suppose that $T$ is a tree of height $\omg$, $\textbf C$ is a ladder system and  $\textbf f$ is  a colouring of $\textbf C$. A $T$-uniformization of $\textbf f$ is a map $\varphi$ on a subtree $\dom \varphi=S$ of $ T$ so that for all $s\in S$ of limit height $\alpha<\omg$ and almost all $\xi\in C_\alpha$,  $$\varphi(s\uhr \xi)=f_\alpha(\xi).$$ 
\end{dfn}

In the above situation, we say that \emph{$\varphi$ uniformizes $\mathbf f$ (on $S$)}. In Figure \ref{diag:unif}, we aimed to emphasize that $\varphi$ is only required to agree with the local colouring $f_\alpha$ along those branches that are bounded in $S$. One can imagine this map $\varphi$ as a coherent collection of countable maps that aim to approximate a uniformization map defined on $\omg$. Indeed, if $S$ does have an uncountable branch then it defines a uniformization on $\omg$ in the classical sense.


 \begin{figure}[H] 
 \centering
 \psscalebox{0.8 0.8} 
{
\begin{pspicture}(0.4,-4.0495653)(11.01,4.0495653)

\rput[bl](10.6,3.5504348){\LARGE{$T$}}
\rput[bl](1.1652174,3.6895652){\LARGE{$\omega_1$}}

\psline[linecolor=black, linewidth=0.04](3.4,2.5504348)(4.8,-2.8495653)(9.2,-2.8495653)(10.2,2.5504348)
\psline[linecolor=black, linewidth=0.04, arrowsize=0.05291667cm 2.0,arrowlength=1.4,arrowinset=0.0]{->}(2.6,-2.6495652)(2.6,3.5504348)


\rput[bl](1.173913,1.8547826){\LARGE{$\alpha$}}
\psline[linecolor=black, linewidth=0.04, linestyle=dashed, dash=0.17638889cm 0.10583334cm](2.2,1.5504348)(11.0,1.5504348)

\psline[linecolor=blue, linewidth=0.04, linestyle=dashed, dash=0.17638889cm 0.10583334cm](2.2,-2.4495652)(2.2,-1.4495652)
\psline[linecolor=blue, linewidth=0.04, linestyle=dashed, dash=0.17638889cm 0.10583334cm](2.2,-0.8495652)(2.2,0.15043478)
\psline[linecolor=blue, linewidth=0.04, linestyle=dashed, dash=0.17638889cm 0.10583334cm](2.2,0.55043477)(2.2,1.1504347)

\rput[bl](1.3,-4.0495653){\textcolor{blue}{\LARGE{$f_\alpha:C_\alpha\to \omega$}}}

\psline[linecolor=blue, linewidth=0.04](2,-2.4495652)(1.6,-2.0495653)(2,-1.8495653)(2,-1.4495652)
\psline[linecolor=blue, linewidth=0.04](2,-0.8495652)(2,-0.24956521)(1.8,-0.04956522)
\psline[linecolor=blue, linewidth=0.04](1.8,0.55043477)(1.8,0.9504348)(2,1.1504347)

\rput[bl](4.3565216,3.3765218){\textcolor{red}{\LARGE{$\varphi:S\to\omega$}}}

\psbezier[linecolor=red, linewidth=0.04](4.4,2.9504347)(4.4,1.9133978)(4.8,-2.6495652)(6.8,-2.649565217391305)(8.8,-2.6495652)(9.0,1.9133978)(9.0,2.9504347)
\psdots[linecolor=red, dotsize=0.3](6.0,1.5504348)
\psdots[linecolor=red, dotsize=0.2](7.0272727,1.5686166)
\psdots[linecolor=black, dotstyle=o, dotsize=0.3, fillcolor=white](8.0,1.5504348)

 \psline[linecolor=red, linewidth=0.04](7.0,-2.4495652)(6.8,-1.4495652)
 \psline[linecolor=black, linewidth=0.04](6.8,-1.4495652)(6.8,-0.8495652)
 \psline[linecolor=red, linewidth=0.04](6.8,-0.8495652)(6.0,0.15043478)
 \psline[linecolor=black, linewidth=0.04](6.0,0.15043478)(6.0,0.55043477)
 \psline[linecolor=red, linewidth=0.04](6.0,0.55043477)(6.0,1.1504347)
 \psline[linecolor=black, linewidth=0.04](6.0,1.1504347)(6.0,1.46)

  \psline[linecolor=red, linewidth=0.04](6.8148465,-2.4307077)(6.3255763,-2.14678)(6.6608396,-1.8508095)(6.5581684,-1.4642106)
 \psline[linecolor=red, linewidth=0.04](6.5347457,-0.91708434)(6.0045314,-0.27544498)(5.622701,-0.24587588)
 \psline[linecolor=red, linewidth=0.04](5.5914893,0.5461795)(5.5914893,0.94617945)(5.791489,1.1461794)

\psline[linecolor=black, linewidth=0.04](6.8,-0.8495652)(7.52,-0.12956522)(7.52,0.8304348)(7.9,1.45)

\psline[linecolor=red, linewidth=0.04](6.46087,-0.4147826)(6.8,0.15043478)
\psline[linecolor=black, linewidth=0.04](6.8028083,0.12312077)(6.8062825,0.59593064)
\psline[linecolor=red, linewidth=0.04](6.8,0.55952567)(6.7909093,1.1686167)
\psline[linecolor=black, linewidth=0.04](6.7909093,1.1686167)(7.0,1.5140711)

\psline[linecolor=red, linewidth=0.04](6.464278,0.51339835)(6.4634094,0.92549956)(6.590247,1.131873)
\psline[linecolor=red, linewidth=0.04](6.4104767,-0.1938992)(6.4666247,0.07317074)(6.677621,0.16747934)

\end{pspicture}
}
    \caption{Uniformization on trees}
  \label{diag:unif}
   \end{figure}
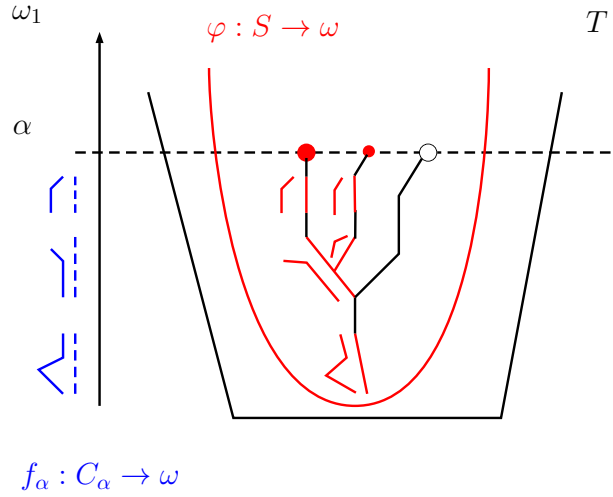

Prior to Moore's and our work, a similar theme of 'uniformization on trees' was investigated by Zoran Spasojevic \cite{treeladder}, in connection to tree topologies, but mostly for a restricted class of colourings that correspond to topological separation axioms.
\medskip

The goal of our project was to contrast the theory of $T$-uniformizations and $\omg$-uniformizations. 
First, note that any subtree of $\omg$ (in our convention of a subtree) must be $\omg$ itself, and so $\dom \varphi=\omg$ for any $\omg$-uniformization $\varphi$. We will say that a $T$-uniformization $\varphi$ of some ladder system colouring is \emph{full} if $\dom \varphi=T$. It is easily seen that if a ladder system colouring has an $\omg$-uniformization then it has a full $T$-uniformization for any tree $T$ of height $\omg$ (and the global colouring on $T$ can be chosen so that it only depends on the levels). In turn, positive results from the classical theory of uniformizations extend to positive results on $T$-uniformizations.

However, there is great flexibility in picking very different domains for $T$-uniformizations of different colourings $\textbf f$. Furthermore, given a $T$-uniformization $\varphi$ and some level $\alpha$ of the tree, the number of $\xi\in C_\alpha$ so that $\varphi(s\uhr \xi)\neq f_\alpha(\xi)$ generally depends on the choice of $s\in S_\alpha$. Thus, it seems (and actually is) harder to negate the existence of $T$-uniformizations.

 
 

\medskip


 Given a ladder system $\textbf C$ on $\omg$, an \emph{$n$-colouring of $\textbf C$} is a sequence of maps $f_\alpha:C_\alpha\to n$.\footnote{We will use this definition for $n=2$ and $n=\oo$.} We say that  $\textbf f$ is \emph{monochromatic} if each  $f_\alpha$ is constant. Now, we will write

\begin{itemize}
 \item $\unif{n}{T}{\textbf C}$ if any $n$-colouring of $\textbf C$ has a $T$-uniformization, and
 \item $\unifc{n}{T}{\textbf C}$ if all monochromatic $n$-colouring of $\textbf C$ has a $T$-uniformization.
\end{itemize}

\medskip

Let us briefly review some fundamental results about ladder system uniformizations: the study was initiated by Saharon Shelah in the 1970s, as he isolated the uniformization property as the combinatorial essence of the Whitehead problem (after solving the problem in \cite{shelah1974infinite}). First, the two extreme cases are summarized below.

\begin{enumerate}[(I)]
 \item $\textmd{MA}_{\aleph_1}$ implies  $\unif{\oo}{\omg}{\textbf C}$  for any ladder system $\textbf C$ \cite{shelah2017proper}, and
 \item\label{it:dsh}  $2^{\aleph_0}<2^{\aleph_1}$ implies $\neg \unifc{2}{\omg}{\textbf C}$ for any ladder system $\textbf C$ \cite{devlin1978weak}.
\end{enumerate}

 We should mention that for any ladder system colouring, there is a (proper) poset which introduces an $\omg$-uniformization but which adds no new reals. However, the latter result implies that these posets cannot be iterated without adding reals, and hence the uniformization property is one of the fundamental barriers to the consistency of forcing axioms compatible with the CH, at least for many reasonably large classes of posets.

At this point, the following theorem of Moore on $T$-uniformizations might come as a surprise:

\begin{enumerate}[(I)]
\setcounter{enumi}{2}
 \item\label{it:moore} it is consistent with CH that $\unif{\oo}{T}{\textbf C}$ holds for any Aronszajn tree\footnote{I.e., a tree of height $\omg$ with countable levels but no uncountable chains} $T$ and ladder system $\textbf C$ \cite[Theorem 1.9]{justinmin}.
\end{enumerate}

As earlier said, Moore's main motivation was to study uncountable linear orders, and he proved that in any model as above, the only minimal uncountable linear orders are $\omg$ and its reverse $-\omg$ \cite[Lemma 3.3]{justinmin}. 
Answering a question of James Baumgartner, this result was later extended by the present author to show that it is consistent with CH that there is a Suslin tree\footnote{A Suslin tree is an Aronszajn tree with no uncountable antichains.} $R$ and $\unif{\oo}{T}{\textbf C}$ holds for all $\textbf C$ and any Aronszajn tree $T$ that embeds no derived subtree of $R$ \cite[Theorem 2.1 and Corollary 2.3]{soukup2018model}. The latter was applied to show that the existence of a Suslin tree does not imply that there are minimal uncountable linear orders beside $\pm \omg$.
The author have used uniformizations on trees to study \emph{strongly surjective linear orders} \cite{stronglysurj} and ladder systems on trees to study graph chromatic number problems \cite{soukuptrees}.

\medskip

Initially, our motivation was to understand if Moore's  model can contain any Suslin trees, or if our latter theorem is optimal and CH does not allow uniformization on Suslin trees. Upon answering this question, we extended various classical results to uniformizations on trees but also found several unexpected positive uniformization results. The results of this manuscript will appear in two separate papers \cite{treeunifA,treeunifB}.



\subsubsection*{Colourings from (weak) diamonds} 
First, in Section \ref{sec:nonunif}, we show that  $2^{\aleph_0}<2^{\aleph_1}$ implies $\neg \unifc{2}{T}{\textbf C}$ for all Suslin trees $T$. In turn, Moore's model from point (\ref{it:moore}) contains no Suslin trees, and our \cite[Theorem 2.1 and Corollary 2.3]{soukup2018model} are optimal.

In Section \ref{sec:nonunif2}, we show how variations of the diamond principle imply  that there are ladder system colourings without $T$-uniformizations for arbitrary Aronszajn trees. As expected, stronger diamonds imply that we can take care of more trees and find even monochromatic 2-colourings without $T$-uniformizations.  In fact, we prove that the parametrized weak diamond $\diamondsuit^\omg(\textmd{non}(\mc M))$ implies $\neg \unif{2}{T}{\textbf C}$ for any $\aleph_1$-tree $T$ and ladder system $\textbf C$. Moreover, $\diamondsuit$ implies that for any Aronszajn tree $T$ there is a ladder system $\textbf C$ so that  $\neg \unifc{2}{T}{\textbf C}$. To complete the picture, we present a model of the GCH (and in fact, of $\diamondsuit$) in which for any $\aleph_1$-tree $T$ and any ladder system $\textbf C$, $\neg \unifc{2}{T}{\textbf C}$ holds. 

Returning to Suslin trees, we show that a natural ccc poset which introduces a uniformization for a given colouring preserves all Suslin trees. Hence, for any Suslin tree $S$, the restricted forcing axiom $MA(S)$\footnote{$MA(S)$ is a weak version of Martin's Axiom. We only require the existence of generic filters for posets which preserve the fixed Suslin tree $S$.} implies that any ladder system colouring has an $\omg$-uniformization. This is done in Section \ref{sec:forceunif}.

\subsubsection*{Constructing trees and uniformizations} 

 In Section \ref{sec:zfc}, we prove one of our  main results: if $\ssqt$ is the tree of all well ordered subsets of $\mb Q$ with a maximum then any ladder system colouring has a $\ssqt$-uniformization; moreover, this fact will be witnessed by a \emph{single master colouring} $h:\ssqt\to \omega$ regardless of the choice of ladder system and its colouring. 
One might say that this is not so surprising as the tree $\ssqt$ is far from Aronszajn: although it has no uncountable chains, the levels are of size continuum and furthermore, $\ssqt$ satisfies rather strong closure properties. We present another variation, still using ZFC only but the larger non-special tree $\sqt$, to ensure the existence of uniformizations defined on non-special subtrees. The latter construction is reminiscent of the techniques applied in \cite{soukuptrees}.

Next, we  show that $\diamondsuit^+$ implies that for any ladder system $\textbf C$, there is an Aronszajn tree $T$ so that $\unifc{\oo}{T}{\textbf C}$ holds. The tree $T$ can either be made special or alternatively, we can ensure that any colouring of $\textbf C$ has a $T$-uniformization defined on a Suslin subtree of $T$. We present these constructions in Section \ref{sec:diamondunif}, and reflect on the sizes of antichains in such trees in Section \ref{sec:largeac}.


\medskip

Our paper is concluded with some open problems and remarks in Section \ref{sec:problems}. Let us also mention that the results of this article are summarized in two lecture videos that are accessible on the author's webpage\footnote{Please see here \url{http://www.logic.univie.ac.at/~soukupd73/}} and, in fact, Theorem \ref{thm:dmplus} is presented in detail.







\subsection*{Notations and terminology} We use standard notations from set theory following classical textbooks \cite{kunen}. In forcing arguments, stronger conditions are smaller.

For us, a tree $(T,<_T)$ is a partially ordered set so that $t^{\downarrow}=\{s\in T:s<_T t\}$ is well ordered for any $t\in T$; we will usually omit the subscript from $<_T$ if it leads to no confusion. Let $\htt(t)$ denote the order type of $t^{\downarrow}$, and let $T_\alpha$ denote the $\alpha$th level of $T$ i.e., all elements $t\in T$ with  $\htt(t)=\alpha$ (and let $T_{<\beta}=\cup_{\alpha<\beta}T_\alpha$). Given $t\in T_\beta$ and $\alpha<\beta$, we let $t\uhr \alpha$ denote the unique element of $T_\alpha$ below $t$. Since all the subtrees $S\subs T$ we work with are downward closed, the levels $S_\alpha$ are just $S\cap T_\alpha$.

An $\aleph_1$-tree is a tree of height $\omg$ with countable levels. An \emph{Aronszajn tree} is an  $\aleph_1$-tree with no uncountable chains, and a \emph{Suslin tree} is an  $\aleph_1$-tree without uncountable chains or uncountable antichains.

The assumption $\diamondsuit$ denotes the existence of a sequence $W=(W_\alpha)_{\alpha<\omg}$ so that for any $X\subs \omg$, the set $\{\alpha<\omg:X\cap\alpha=W_\alpha\}$ is stationary. Similarly, $\diamondsuit^+$ asserts the existence of a sequence $W=(W_\alpha)_{\alpha<\omg}$ so that $|W_\alpha|\leq \oo$ and for any $X\subs \omg$, there is a closed unbounded $D\subset \omg$ so that  for any $\alpha\in D$,  $X\cap \alpha,D\cap \alpha\in W_\alpha.$ 


We will use the following standard fact multiple times.

\begin{fact}\label{fact:suslin}
Suppose that $S$ is an $\aleph_1$-tree and $M$ is a countable elementary submodel of some $H(\Theta)$ with $S\in M$. Let $Y\subs S$ be an element of $M$ and $t\in Y\setm M$. Then

\begin{enumerate}
    \item if $S$ has no uncountable antichains then  $ Y\cap M\cap t^\downarrow\neq \emptyset$, and
    \item if $S$ has no uncountable chains then  $Y\cap M\setm t^\downarrow\neq \emptyset$.
\end{enumerate}
\end{fact}
\begin{cproof}(1) Indeed, take a maximal antichain $A\subs Y$ so that $A\in M$, and note that $A\subs M$ as $A$ is countable. Then $\vareps=\sup \htt[A]+1\in M$ and $M\models t\uhr \vareps$ is compatible with some element of $Y$. So there must be some $s\in A$ so that $s$ is compatible with $t\uhr \vareps$ and so $s\leq t\uhr \vareps\leq t$.

(2) If we choose $A\subs Y$ to be a maximal chain in $A$ then again, $A$ is countable and  we must have some element of $A$ off the branch $t^\downarrow$.
\end{cproof}

\smallskip

\subsection*{Acknowledgments} We thank S.-D. Friedman and M. Levine for valuable discussions and helpful comments on some of our results.
The author would like to thank the generous support of the FWF Grant I1921 and NKFIH OTKA-113047.

\section{Non-uniformization results for Suslin trees}\label{sec:nonunif}

In this section, we will show that if $2^{\aleph_0}<2^{\aleph_1}$ then for any $\aleph_1$-tree $T$ with no uncountable antichains and any ladder system \textbf C, $\unifc{2}{T}{\mathbf C}$ \emph{fails}.


Let us mention that in \cite{justinmin}, where $T$-uniformizations were first introduced, the only non-uniformization result that can be immediately extracted is \cite[Lemma 3.3]{justinmin}: under $2^{\aleph_0}<2^{\aleph_1}$, any Aronszajn tree $T$ that is club-embeddable into all its subtrees\footnote{I.e., for any subtree $S\subseteq T$ there is a club $D\subs \omg$ and an order-preserving injection $\bigcup_{\alpha\in D}T_\alpha\to \bigcup_{\alpha\in D}S_\alpha$.} and any ladder system $\mathbf C$, there is a monochromatic 2-colouring of $\mathbf C$ without a $T$-uniformization. Such trees can be constructed by forcing as done by Abraham ans Shelah \cite{abrahamiso}, or by using $\diamondsuit^+$ to construct a minimal lexicographically ordered Aronszajn tree \cite{baumorder,stronglysurj} and then invoke \cite[Lemma 2.9]{justinmin}.

\medskip

We start by recalling some definitions and in particular the basis of the result of Devlin and Shelah from point (\ref{it:dsh}).

\begin{thm}\cite{devlin1978weak}
$2^{\aleph_0}<2^{\aleph_1}$ if and only if for any $F:2^{<\omg}\to 2$ there is a $g\in 2^\omg$ so that for any $f\in 2^\omg$ the set $\{\alpha\in \omg: g(\alpha)\neq F(f\uhr \alpha)\}$ is stationary.
\end{thm}

The latter statement is often referred to as \emph{the weak diamond}, denoted by $\Phi^2_\omg$, which was discovered in search for the minimal assumptions that imply $\neg\unifc{2}{\omg}{\mathbf C}$. In general, applications of this principle usually involve coding a countable structure into a countable sequence of 0s and 1s to define $F$. In fact, we will use the following form of $\Phi^2_\omg$  where $H(\Theta)$ denotes the collection of sets of hereditary cardinality $<\Theta$. 

\begin{thm}\label{thm:code}\cite[Theorem 3.2]{justinmin}
$2^{\aleph_0}<2^{\aleph_1}$ implies that for any  $F:H(\aleph_1)\to 2$ there is a $g:\omg\to 2$ so that for any $ U\in H(\aleph_2)$ there is a countable elementary submodel $M\prec H(\aleph_2)$ containing $ U$  so that $g(\omg\cap M)\neq F(U^M)$.\footnote{$ U^M$ is the image of $U$ under the transitive collapse of $M$. Typically, this will coincide with $ U\cap M$. Furthermore, we can always assume that $M$ contains whatever countably many parameters we require form $H(\aleph_2)$, not just $U$.}

\end{thm}

The function $g$ from the theorem will be referred to as an \emph{oracle} for $F$.

\medskip

By Moore's result from point (\ref{it:moore}), $\Phi^2_\omg$ by itself is not strong enough to  show the existence of colourings without a $T$-uniformization for an arbitrary  Aronszajn tree $T$. However, $\Phi^2_\omg$ does have consequences for Suslin trees.



\begin{thm}\label{thm:nonunifs}
 $\Phi^2_\omg$ implies  $\neg \unifc{2}{T}{\textbf C}$ for any $\aleph_1$-tree $T$ without uncountable antichains and any ladder system $\textbf C$. 
\end{thm}

This result generalizes Devlin and Shelah's theorem from point (\ref{it:dsh}), but applies to Suslin trees as well. So, we see that Moore's model cannot contain Suslin trees, and the current author's result from \cite{soukup2018model} is optimal.  Let us also mention that $\Phi^2_\omg$ easily implies that for any ladder system $\textbf C$ and Aronszajn tree $T$ there is a monochromatic 2-colouring of $\textbf C$ without a \emph{full} $T$-uniformization i.e., a $T$-uniformization $\varphi$ with $\dom \varphi=T$.

\begin{tproof}
 Fix an $\aleph_1$-tree $T$ without uncountable antichains and ladder system $\textbf C$; we can assume that $T\subs \oo^{<\omg}$ and so $T_{<\alpha}\in H(\aleph_1)$ for all $\alpha<\omg$.
 
 We need to find a monochromatic 2-colouring of $\mathbf C$ with no $T$-uniformization, and  to do so, we first define a function $F:H(\aleph_1)\to 2$. In fact, we will only specify the values of $F$ on maps $\varphi:S\to 2$ in $H(\aleph_1)$ where  $S$ is a subtree of $T_{<\alpha}$ for some limit $\alpha<\omg$. On all other elements of $H(\aleph_1)$, we let $F$ be constant 1.  
   
  Two cases are distinguished.
   
   \textbf{Case 1.} Suppose that there is an $s\in S$ and $i<2$, so that for any $t\in T_\alpha$ above $s$ such that $t^\downarrow\subs S$, the set $$\{\xi\in C_\alpha:\varphi(t\uhr \xi)=i\}$$ is cofinal in $\alpha$.  Pick such an $s,i$ canonically (using some well order of $H(\aleph_1)$), and define $F(\varphi)=i$. 
   
   
   \textbf{Case 2.} If there is no such pair $s,i$ then we let $F(\varphi)=1$.
   
   Now, let $g\in 2^\omg$ be the oracle for $F$ given by $\Phi^2_\omg$. We claim that there is no $T$-uniformization for the monochromatic coloring $\textbf f$ that is defined by taking $f_\alpha$ to be constant $g(\alpha)$.
   
   Otherwise, let $S\subs T$ be a subtree and suppose $\varphi:S\to 2$ is a $T$-uniformization of $\textbf f$. Now, we reach a contradiction: find a countable elementary submodel $M\prec H(\aleph_2)$ containing $T,\varphi$ and $g$ so that $g(\alpha)\neq F(\varphi\cap M)$. Let $\alpha=M\cap \omg$.
   
   
   \begin{clm}
    We used the second case definition for $F(\varphi\cap M)$.
   \end{clm}
\begin{cproof}

   Otherwise, we used the first case definition of $F$ with some node $s\in S_{<\alpha}$ and $i<2$. But then for any  $t\in T_\alpha$ above $s$ so that $t^\downarrow\subs S$, $\varphi(t\uhr \xi)=i$ for infinitely many $\xi\in C_\alpha$. However, $g(\alpha)=1-i$ and so for every $t\in S_\alpha$ such that $t^\downarrow\subset S$, $\varphi(t\uhr \xi)=1-i$ holds for almost all $\xi\in C_\alpha$. In turn, there cannot be any element of $S\cap T_\alpha$ above $s$, which contradicts that $S$ was pruned.
 
\end{cproof}

 In particular, we deduce that $F(\varphi\cap M)=1$ and so $g(\alpha)=0$.  We know that Case 1 failed with any $s\in S_{<\alpha}$ and $i=0$, so we can find $t_0\in T_\alpha$ so that $t_0^\downarrow\subs S_{<\alpha}$ and $\varphi(t_0\uhr \xi)=1$ for almost all $\xi\in C_\alpha$. 
 Consider the set $$X=\{t\in T:\htt(t)=\beta, g(\beta)=0, t^\downarrow\subs S \textmd{ and } \varphi(t\uhr \xi)=1 \textmd{ for almost all }\xi\in C_\beta\}.$$
 
 Note that $X\in M$ and $t_0\in X$ so by Fact \ref{fact:suslin}, we can find some $t\in t_0^\downarrow\cap X$. However, any $t\in t_0^\downarrow$ is also in $S$, so $\varphi(t\uhr \xi)=g(\beta)=0$ for almost all $\xi\in C_\beta$ (where $\beta=\htt(t)$). This contradicts that $t\in X$ (i.e., that $\varphi(t\uhr \xi)=1 \textmd{ for almost all }\xi\in C_\beta\}$), which in turn finishes the proof of the theorem.

\end{tproof}
 In Section \ref{sec:forceunif}, we show that even $\unif{\oo}{\omg}{\textbf C}$ for all $\textbf C$ is consistent with the existence of Suslin trees (but necessarily, CH will fail in such a model).  However, we are not sure if it suffices to assume in the above result that $T$ has no stationary antichains.

\section{Non-uniformization results for Aronszajn trees}\label{sec:nonunif2}

Our next goal is to find an assumption that gives ladder system colourings which have no $T$-uniformization for any $\aleph_1$-tree $T$ that is not necessarily Suslin. To prove such a result, we employ the technique of parametrized weak diamonds \cite{DHM,osvaldo,minami2008suslin}.  Let $\mc M$ stand for the ideal of meager sets in $2^\oo$.
  
  \begin{dfn}
   Let $\diamondsuit^{\omg}(\textmd{non}(\mc M))$ denote the following statement: if $X$ is an $\omg$ set of ordinals and $F:\bigcup_{\alpha<\omg} \alpha^\alpha\to \mc M$ so that $F\uhr \alpha^\alpha\in L(\mathbb R)[X]$ for all $\alpha<\omg$, then there is a $g:\omg\to 2^\oo$ so that for all $f:\omg\to \omg$ the set $$\{\alpha<\omg:f\uhr \alpha\in \alpha^\alpha \textmd{ and } g(\alpha)\notin F(f\uhr \alpha)\}$$ is stationary.
  \end{dfn}
  
  Recall that $L(\mathbb R)$ is the class of sets constructible from $\mb R$ (in the sense of G\"odel) and $L(\mathbb R)[X]$ is the minimal model extending $L(\mathbb R)$ which contains $X$. See \cite[Chapter 13]{jech} for more details on constructibility. So, compared to  $\Phi^2_\omg$, we changed the range of $F$ from 2 to $\mc M$, required $F$ to be nicely definable and the oracle $g$ will stationarily often avoid the 'bad' set given by $F$.


$\diamondsuit^{\omg}(\textmd{non}(\mc M))$ implies the existence of Suslin trees \cite[Theorem 3.1]{pardiamond} and follows from  the classical diamond principle $\diamondsuit$.\footnote{Recall that $\diamondsuit$ says that there is a sequence  $W=(W_\alpha)_{\alpha<\omg}$ so that for any $X\subs \omg$, the set  $\{\alpha<\omg:X\cap \alpha= W_\alpha\}$ is stationary.} However, $\diamondsuit^{\omg}(\textmd{non}(\mc M))$ has the great advantage that it can hold in various models with $2^{\aleph_0}=2^{\aleph_1}$ (see \cite{osvaldo} for more details). 

%


\begin{thm}\label{thm:weakdnonunif}
$\diamondsuit^{\omg}(\textmd{non}(\mc M))$ implies $\neg \unif{2}{T}{\textbf C}$ for any $\aleph_1$-tree $T$ and ladder system $\textbf C$.

%
\end{thm}
In turn,  $\diamondsuit^{\omg}(\textmd{non}(\mc M))$  is inconsistent with quite weak forms of uniformization on trees, but we will see in later sections that, quite surprisingly, we cannot always produce a \emph{monochromatic} 2-colouring without a $T$-uniformization from diamond-like assumptions. Moreover, T. Yorioka recently proved\footnote{Personal communication.} that the analogous principle $\diamondsuit^{\omg}(\textmd{cov}(\mc N))$ is consistent with $\unif{2}{\omg}{\textbf C}$ (modulo some large cardinal assumptions).


\begin{tproof}
  Fix $T$ and $\textbf{C}$, and let $(\alpha_k)_{k\in \oo}$ denote the increasing enumeration of $C_\alpha$. We will first define $F:\bigcup_{\alpha<\omg} \alpha^\alpha\to \mc M$ using $T$ as a parameter. By a standard coding argument (much like in Theorem \ref{thm:code}), we can associate (in a canonical way) to any  $\varphi:S\to 2$ so that $S\subseteq T_{<\alpha}$ is a subtree and $\alpha\in \lim\omg$ some code $f_\varphi:{\beta_\varphi}\to{\beta_\varphi}$. Moreover, for club many $\alpha\in \omg$ the code $f_\varphi$ of any $\varphi$ of height $\alpha$ will be defined on the same $\alpha$ (i.e., $\beta_\varphi=\alpha$).

  Now, suppose $f\in \alpha^\alpha$ codes some $\varphi:S\to 2$ so that $S\subseteq T_{<\alpha}$ is a subtree and $\alpha\in \lim\omg$. Then we let $$F(\varphi)=\{\psi\in 2^{\oo}:(\exists t\in T_\alpha)t^\downarrow\subs S \textmd{ and }\varphi(t\uhr \alpha_k)=\psi(k)\textmd{ for almost all }k<\oo\}$$
 $$=\bigcup_{t\in T_\alpha,t^\downarrow\subs S}\;\bigcup_{n\in \oo}\;\bigcap_{k\geq n}\{\psi\in 2^{\oo}:\varphi(t\uhr \alpha_k)=\psi(k)\}. $$
 
 Otherwise, we let $F(f)=\emptyset$; it is easy to see that $F\uhr \alpha^\alpha\in L(\mb R)[T]$.

 \begin{clm}
  $F(\varphi)$  is a meager subset of $2^{\oo}$.
 \end{clm}
 \begin{cproof}
  It suffice to note that the set $$\bigcap_{k\geq n}\{\psi\in 2^{\oo}:\varphi(t\uhr \alpha_k)=\psi(k)\}$$ is nowhere dense for any choice of $t\in T_\alpha$ and $n<\oo$.
 \end{cproof}

  Now, suppose $g:\omg\to 2^\oo$ is the oracle for $F$ witnessing $\diamondsuit^{\omg}(\textmd{non}(\mc M))$, and define the colouring $f_\alpha:C_\alpha\to 2$  by $f_\alpha(\alpha_k)=g(\alpha)(k)$ for $k<\oo$. We claim that this (not necessarily monochromatic) 2-colouring $\textbf f=(f_\alpha)_{\alpha<\omg}$ has no $T$-uniformization.
 
 Otherwise, suppose that $\varphi:S\to 2$ is a $T$-uniformization. Now, we can find some limit $\alpha<\omg$ so that   $g(\alpha)\notin F(\varphi \uhr S_{<\alpha})$. This means that for any $t\in S_\alpha\subseteq T_\alpha$, $$\varphi (t\uhr \alpha_k) \neq g(\alpha)(k)=f_\alpha(\alpha_k)$$ for infinitely many $k$, a contradiction.

 \end{tproof}

Next, one would like to find a \emph{monochromatic 2-colouring without a $T$-uniformization}, much like in the case of Suslin trees. We will use $\diamondsuit$ to do this but if the interested reader looks closely at the preceding proofs for an optimal assumption then he or she will be lead to a fairly technical common strengthening of  $\Phi^2_\omg$ and $\diamondsuit^{\omg}(\textmd{non}(\mc M))$ that suffices for the same argument to work;\footnote{Let $\diamondsuit^{\omg}(2\times \textmd{non}(\mc M))$ stand for the following: if $X$ is an $\omg$ set of ordinals and $F:\bigcup_{\alpha<\omg} \alpha^\alpha\to 2\times \mc M$ so that $F\uhr \alpha^\alpha\in L(\mathbb R)[X]$ for all $\alpha<\omg$, then there is a $g:\omg\to 2\times 2^\oo$ so that for all $f:\omg\to \omg$ the set $\{\alpha<\omg:f\uhr \alpha\in \alpha^\alpha \textmd{ and }$  $F(f\uhr \alpha)=(i,\mc C)$ implies $g(\alpha)=(1-i,c)$ where $c\notin \mc C\}$ is stationary.  $\diamondsuit^{\omg}(2\times \textmd{non}(\mc M)$ follows from $\diamondsuit$ and, most likely, $\diamondsuit^{\omg}(2\times \textmd{non}(\mc M))$ is also consistent with the negation of CH.} however, we chose to present the simpler $\diamondsuit$-based argument here as we see no further application or real advantage of the other, more technical approach at this point.

 
 \begin{thm}
   $\diamondsuit$ implies that for any $\aleph_1$-tree $T$, there is a ladder system $\mathbf C$ so that $\neg \unifc{2}{T}{\textbf C}$.
 \end{thm}
 
 Note that, given a tree $T$, we are choosing both the ladder system and the monochromatic 2-colouring without $T$-uniformization; we will see later that this is best possible from diamond-type assumptions.

We use the following equivalent form of $\diamondsuit$ (see \cite{pardiamond}): for any $F:H(\aleph_1)\to 2^\oo$ there is a $g:\omg\to 2^\oo$ so that   for any $U\in H(\aleph_2)$ there is a countable elementary $M\prec H(\aleph_2)$ containing $U$ so that $F(U^M)=g(M\cap \omg)$.

 \begin{tproof}
 We fix an $\aleph_1$-tree $T$ and define a map $F:H(\aleph_1)\to 2^\oo$ first as follows. Take bijections $e_\alpha:\oo\to \alpha$ for all $\alpha<\omg$.  Given $\varphi:S\to 2$ for some pruned $S\subseteq T_{<\alpha}$, we distinguish two cases.
 
    \textbf{Case 1.}  Suppose that there is an $s\in S$ and $i<2$ so that for any $t\in T_\alpha$ above $s$ such that $t^\downarrow\subs S$, the set $$I_t=\{\xi<\alpha:\varphi(t\uhr \xi)=i\}$$ is cofinal in $\alpha$.  Pick such an $s,i$ canonically (using some well order of $H(\aleph_1)$), and define $F(\varphi)(0)=1-i$.    
    
    Now, let $B_s=\{t\in T_\alpha:t^\downarrow\subs S, s\leq t\}$. We choose the values $F(\varphi)(k)$ for $k\geq 1$ so that 
    \begin{enumerate}
        \item $E_{\varphi}=\{e_\alpha(k):k\geq 1, F(\varphi)(k)=1\}$ is a cofinal, type $\oo$ subset of $\alpha$, and
        \item for any $t\in B_s$, there are infinitely many $\xi\in E_\varphi$ such that $\varphi(t\uhr \xi)=i$.
    \end{enumerate}
   
   In plain words, we use $F(\varphi)(0)$ to record a colour which contradicts $\varphi$ on each branch $t^\downarrow$ restricted to the levels in $E_\varphi$ (infinitely often). The set  $E_\varphi$ is recorded by the restriction $F(\varphi)\uhr[1,\omega)$.
      \medskip
      
  \textbf{Case 2.} If Case 1 fails, we let $F(\varphi)$ be the constant 1 function.
    \medskip
    
  On any other element of $H(\aleph_1)$, we define $F$ to be constant 1 again.
  
  \medskip
  
  Now, let $g:\omg\to 2^\oo$ be the oracle for $F$ that witnesses    $\diamondsuit$. We define the ladder system $\mathbf C$: if  $\{e_\alpha(k):k\geq 1, g(\alpha)(k)=1\}$ has type $\omega$ and is cofinal in $\alpha$ then we let this set be $C_\alpha$. Otherwise, $C_\alpha$ is any type $\oo$, cofinal subset of $\alpha$. To colour $\mathbf C$, we simply let $f_\alpha:C_\alpha\to 2$ be constant $g(\alpha)(0)$.
  \medskip

  We claim that $\mathbf f$ has no $T$-uniformization. Otherwise, let  $\varphi:S\to 2$ be a $T$-uniformization of $\textbf f$. Now, find a countable elementary submodel $M\prec H(\aleph_2)$ containing $T$ and $\varphi$ such that $g(\alpha)=F(\varphi\cap M)$. Let $g(0)=1-i=F(\varphi\cap M)(0)$ and $\alpha=\omg\cap M$.
  
   
   \begin{clm}
    We used the second case definition for $F(\varphi\cap M)$.
   \end{clm}
      
\begin{cproof}

   Otherwise, we used the first case definition of $F$ with some node $s\in S_{<\alpha}$ and $i<2$. But then for any  $t\in T_\alpha$ above $s$ so that $t^\downarrow\subs S$, $\varphi(t\uhr \xi)=i$ for infinitely many $\xi\in E_{\varphi\cap M}= C_\alpha$. However, $g(\alpha)=1-i$ and so for every $t\in S_\alpha$, $\varphi(t\uhr \xi)=1-i$ holds for almost all $\xi\in C_\alpha$. In turn, there is no element of $S\cap T_\alpha$ above $s$, which contradicts that $S$ was pruned.
 
\end{cproof}

 In particular, we deduce that $F(\varphi\cap M)(0)=1$ and so $g(\alpha)=0$.  In turn, by elementarity, there are arbitrary large $\alpha'<\alpha$ so that $g(\alpha')=0$ as well. Hence, for any $t\in T_\alpha$ such that $t^\downarrow\subs S$ and any $\vareps<\alpha$, there is $\xi\in \alpha\setm \vareps$ so that $\varphi(t\uhr \xi)=0$. But this implies that Case 1 does hold for $\alpha$ with $i=0$ and any choice of $s$. This contradiction (from the existence of a $T$-uniformization for $\mathbf f$) finishes the proof.

 \end{tproof}

Finally, we end this section by showing that consistently all trees and all ladder systems fail the uniformization property in the strongest sense.

\begin{thm}\label{thm:nonunif}
After adding $\aleph_2$ Cohen subsets to $\omg$, for any $\aleph_1$-tree $T$ and ladder system $\textbf C$, $\neg \unifc{2}{T}{\textbf C}$.
\end{thm}

That is, we will iterate posets of the form $\{p:\dom p\in \omg, \ran p=2\}$ with countable support in $\omega_2$ stages. Note that this forcing is countably closed and so no new countable sets are introduced. We also use the following observation:

\begin{obs}\label{obs:sigma}
Suppose that $\mb Q$ is a countably closed poset and $\textbf f$ is a ladder system colouring with no $T$-uniformization (for some $\aleph_1$-tree $T$). Then $\textbf f$ has no $T$-uniformization in any forcing extension by $\mb Q$.
\end{obs}
\begin{proof}
Suppose that $q_0\force \dot \varphi:S\to \oo$ is a $T$-uniformization of $\textbf f$. Construct a decreasing sequence of conditions $(q_\alpha)_{\alpha<\omg}$ with maps $(\psi_\alpha)_{\alpha<\omg}$ so that $q_\alpha\force \dot \varphi\uhr S_{<\alpha}=\psi_\alpha$. Note that whether a map $\psi$ is a $T$-uniformization of $\textbf f$ is decided on its countable initial segments. In turn, $\psi=\bigcup_{\alpha<\omg}\psi_\alpha$ must be a $T$-uniformization of $\textbf f$ in the ground model.
\end{proof}

\begin{tproof}[Proof of Theorem \ref{thm:nonunif}]
 Any $\aleph_1$-tree $T$ and ladder system $\textbf C$ of the final extension appears at some intermediate stage of the iteration, so we will assume $T,\textbf C$  are in the ground model $V$ and show that if  $\dot g\in 2^\omg$ is Cohen generic over $V$ then $\neg \unifc{2}{T}{\textbf C}$ holds in $V[\dot g]$. In fact, we prove that if $\dot f_\alpha$ is constant $\dot g(\alpha)$ then $\textbf f$ has no $T$-uniformization in $V[\dot g]$. Since the whole forcing is countably closed, the rest of the iteration cannot add a uniformization for this colouring in later stages by Observation \ref{obs:sigma}.
 
 Now, suppose that $\textbf{f}$ does have a uniformization in $V[\dot g]$ and we reach a contradiction. Find a Cohen condition $p$ (i.e., a countable partial function from $\omg$ to $2$) and name $\dot \varphi$, so that $p\force \dot \varphi:\dot S\to 2$ is a $T$-uniformization of  $\textbf{f}$. Take a continuous, increasing sequence of countable elementary submodels $(M_n)_{n\leq\oo}$ of $H(\aleph_2)$ so that $T,\dot g,\dot \varphi,p\in M_0$. Let $\delta_n=M_n\cap \omg$ for $n\leq\oo$, and list $T_{\delta_\oo}$ as $\{t_n:n\in \oo\}$. We will define $q_0\geq q_1\geq q_2\dots $ conditions below $p$ so that $q_n\in M_{n+1}$, $q_n$ decides $\varphi\uhr \dot S_{<\delta_n}$ and $q_{n+1}\force t_n\uhr \delta_{n+1}\notin \dot S$ and so $t_n\notin \dot S$ by the downward closure of $\dot S$. In turn, any lower bound $q_\oo$ to this sequence will force that $t_n\notin \dot S$ for all $n<\oo$ and so $q_\oo\force \dot S\cap T_{\delta_\oo}=\emptyset$, a contradiction.
 
 Let us describe the general step in the construction i.e., finding $q_{n+1}$ from $q_n$ (where $q_{-1}=p$). Given $q_n\in M_{n+1}$, we first find an extension $r\in M_{n+2}$ so that $r\subs M_{n+1}$ and $r$ is $M_{n+1}$-generic; this can be done since the forcing is $\sigma$-closed. Now, $r$ decides $\varphi\uhr \dot S_{<\delta_{n+1}}$ and decides whether $t_n^\downarrow\cap M_{n+1}\subs \dot S$ or not. In the former case, there is an $i<2$ so that $r$ forces that $\dot \varphi(t_n\uhr \xi)=i$ for infinitely many $\xi\in C_{\delta_{n+1}}$. Since $r\subs M_{n+1}$, we can find an extension $q_{n+1}\leq r$ in $M_{n+2}$ that forces $\dot g(\delta_{n+1})=1-i$. In turn, $q_{n+1}$ forces that $t_n\uhr \delta_{n+1}\notin \dot S$ and so $t_n\notin \dot S$ as well. 
 
\end{tproof}

%

%
%

 \section{Forcing uniformizations and preserving Suslin trees}\label{sec:forceunif}

In \cite{justinmin}, Moore introduced a forcing  that uniformizes a given colouring on an Aronszajn tree $T$, which furthermore can be iterated without adding reals. In turn, consistently, CH  and $\unif{\oo}{T}{\mathbf C}$ holds  for any ladder systems $\mathbf C$ and Aronszajn tree $T$.

However, we proved in Theorem \ref{thm:nonunifs} that whenever CH holds, for any Suslin tree $T$ there are monochromatic 2-colourings of ladder systems which do not have  $T$-uniformization. In turn, there could be no Suslin trees in Moore's model. We also proved  in \cite{soukup2018model} that one can preserve CH and a single Suslin tree $R$, while forcing the existence of $T$-uniformizations for all $\oo$-colourings and those trees $T$ which do not embed $R$ in a strong sense.

\medskip

Our goal in this section is to show that if we do not require CH then even $\unif{\oo}{\omg}{\mathbf C}$ for all $\mathbf C$ is consistent with the existence of Suslin trees. We will do this by proving that for any ladder system colouring $\mathbf f$, there is a natural ccc forcing $\mc P_{\mathbf f}$ which introduces an $\omg$-uniformization for $\mathbf f$ and such that $\mc P_{\mathbf f}$ preserves all Suslin trees. One can then apply well known iteration theorems to find the desired model; the details follow below.


\begin{thm}\label{thm:force}
 For any ladder system $\mathbf C$ and $\oo$-colouring $\mathbf f$ of $\mathbf C$, there is a ccc forcing $\mc P_{\textbf f}$ of size $\aleph_1$,  so that 
 \begin{enumerate}
  \item  $\force_{\mc P_{\mathbf f}}$ there is an $\omg$-uniformization of $\mathbf f$, and
  \item if $R$ is a Suslin tree in the ground model then $\force_{\mc P_{\mathbf f}}$ $\check R$ is Suslin.

 \end{enumerate}

\end{thm}

\begin{cor}\label{cor:force}
Suppose CH holds. Then there is a proper, cardinality and cofinality preserving forcing $\mb P$ so that in $V^{\mb P}$, $\unif{\oo}{\omg}{\mathbf C}$ holds for any ladder system $\mathbf C$, and any Suslin tree in the ground model remains Suslin in $V^{\mb P}$.
\end{cor}

\begin{tproof}
An appropriate countable support iteration (in length $\oo_2$) of posets of the form $\mc P_{\mathbf f}$  will force that $\unif{\oo}{\omg}{\mathbf C}$ holds for all ladder systems $\mathbf C$. The fact that any ground model Suslin tree remains Suslin follows from Tadatoshi Miyamoto's preservation theorem \cite{spres}.\footnote{That is, if successor stages of a proper, countable support iteration preserve a given Suslin tree, then so does the limit steps of the iteration.}
\end{tproof}

Similar models that combine features of the constructible universe $L$ and consequences of MA or PFA have received considerable attention \cite{avraham1978consistency, avraham1982forcing, larson2002katetov, yorioka2010uniformizing, MR2979581}. However, we could not find a reference for a model combining the existence of Suslin trees and the uniformization property.
For example, in \cite{avraham1982forcing}, a forcing axiom for \emph{stable posets} is proved consistent: although the model satisfies $\unif{\oo}{\omg}{\mathbf C}$, it has no Suslin trees. 

More recently, the centre of attention has been on models of $MA(S)$ (or $PFA(S)$) i.e., the forcing axiom for ccc  partial orders (or proper posets, respectively) that preserve a fixed (coherent) Suslin tree $S$. We have the following corollary now.

\begin{cor}Suppose $S$ is a Suslin tree. Then
$MA(S)$ implies $\unif{\oo}{\omg}{\mathbf C}$ for all $\mathbf C$.
\end{cor}

This should be compared with a result of P. Larson and S. Todorcevic: if $S$ is a Suslin tree and $\textbf C$ is an $S$-name for a ladder system then forcing with $S$ will introduce a colouring of $\textbf C$ with no $\omg$-uniformization \cite[Theorem 6.2]{larson2001chain}. It would be interesting to see how much of that argument carries over to $T$-uniformizations.


\medskip

Now, let us prove the theorem. There are various ways to introduce a uniformization for a given ladder system colouring by a ccc forcing. These techniques include posets of
\begin{enumerate}
\item countable partial functions $p$ from $\omg$ to $\oo$ that are defined on finite unions of the ladders, and so that $p$ agrees with $f_\alpha$ almost everywhere on its domain \cite[Chapter II, Theorem 4.3]{shelah2017proper};
    \item finite partial maps $p$ from $\lim(\omg)$ to $\oo$ so that   $$\{f_\alpha\uhr \{C_\alpha(n): n\geq p(\alpha)\}:\alpha\in \dom p\}$$ are pairwise compatible \cite{yorioka2010uniformizing}.
\end{enumerate}

In both cases the order on the posets is extension as functions.\footnote{See \cite{whitehead1,whitehead2} for further examples. Here the uniformization is done by countable approximations which allow certain fusion arguments.} 
It can be shown that both the above posets preserve Suslin trees (see the remarks at the end of the section). However, we  take this opportunity to present another natural variant: 
\begin{enumerate}
\setcounter{enumi}{2}
\item forcing with finite partial maps defined on what we call \emph{$\mathbf C$-closed sets}. In our case, the uniformization property is coded into the ordering. 
\end{enumerate}

We present the details below.

\begin{tproof}[Proof of Theorem \ref{thm:force}] Given a ladder system $\mathbf C$, we say that $D\subs \omg$ is $\textbf C$-closed if for any $\alpha\neq \beta\in D\cap \lim \omg$,  $C_\alpha\cap C_\beta\subs D$. Let $\ecl{E}$ denote the smallest $\textbf C$-closed superset of $E$.

We show that a finite set $E$ has finite $\textbf C$-closure, and even more:

 \begin{obs}\label{clm:cl} Suppose that $\mathbf C$ is a ladder system.
 \begin{enumerate}
 \item Any finite set $E$ has finite $\mathbf C$-closure and $\max \ecl E= \max E$;
  \item any initial segment of a $\mathbf C$-closed set is $\mathbf C$-closed;
  \item\label{it:maxcl} if $D$ is $\mathbf C$-closed and $E$ is finite then $$\ecl{D\cup E}\setm (\max E+1)=D\setm (\max E+1).$$ 
 \end{enumerate}
 \end{obs}

 \begin{cproof}  First, let $E$ be finite and we show, by induction on  $\vareps=\max E$, that $E$ has  finite $\mathbf C$-closure. We can assume that $\vareps\in \lim \omg$ otherwise $\ecl{E}=\ecl{E\cap \vareps}\cup \{\vareps\}$ and we are done by induction. Let $$E^*=(E\cap \vareps) \cup \bigcup_{\alpha\in \vareps \cap E}C_\alpha\cap C_\vareps,$$ and note $\max E^*<\vareps=\max E$  and that $\ecl E=\ecl{E^*}\cup \{\vareps\}$. So $\ecl E$ must have size $|\ecl{E^*}|+1$ which is finite by induction.

 Now suppose that $D$ is $\mathbf C$-closed, let $\gamma<\omg$ and pick any $\alpha<\beta\in D\cap \gamma$. Then $$C_\alpha\cap C_\beta\subs D\cap \beta\subs D\cap \gamma.$$ So the initial segment $D\cap \gamma$ is $\mathbf C$-closed too.
 
 Finally, we prove the last statement by induction on $\delta=\max D$. Since $D\cap \delta$ is $\mathbf C$-closed, we can apply induction to see that $$\ecl{(D\cap \delta)\cup E}\setm (\max E+1)=(D\cap \delta)\setm (\max E+1),$$ and by adding $\delta$ to both sides we get the desired equality.

 \end{cproof}

 \medskip
 
Now, given a ladder system $\mathbf C$ and $\oo$-colouring $\mathbf f$, let $\mc P=\mc P_{\textbf f}$  consist of all finite maps $p:D^p\to \oo$ so that $D^p\subs \omg$ is $\textbf C$-closed. The  extension in $\mc P$ is defined by $q\leq p$ if $q\supseteq p$ and 
 \begin{center}

  \begin{minipage}[c]{0.7\textwidth}
 ($\dagger$) for all $\alpha\in D^p\cap \lim \omg$ and all $\xi\in C_\alpha\cap D^{q}\setm D^p $, $q(\xi)=f_\alpha(\xi).$
  \end{minipage}

 \end{center}


 First, we prove a few simple facts on  $\mc P$ which show that a generic filter gives an $\omg$-uniformization of $\mathbf f$.

  \begin{clm} Suppose that $G\subs \mc P$ is generic and let $\dot \varphi=\cup G$.
 \begin{enumerate}
  \item\label{it:1}    $D_\alpha=\{q\in \mc P:\alpha\in D^q\}$ is dense in $ \mc P$ for all $\alpha\in \omg$;
  \item\label{it:1.5} for any $p\in \mc P$ and $\alpha\in D^p$, $p\force \dot \varphi(\xi)=f_\alpha(\xi)$ for all $\xi\in C_\alpha\setm D^p$;
  
  \item\label{it:2} $\force_{\mc P}\dom \dot \varphi=\omg $ and for all $\alpha\in \lim \omg$ and almost all $\xi\in C_\alpha$, $\force_{\mc P} \dot \varphi(\xi)=f_\alpha(\xi)$.

 \end{enumerate}

 \end{clm}
\begin{cproof} (\ref{it:1}) Given $p$ and $\alpha\in \omg\setm  D^p$, we form $D=\ecl{D^p\cup \{\alpha\}}$. This is a finite set and we need to show that there is some $q\supseteq p$ in $\mc P$ so that $\dom q=D$ i.e., that  the extension property ($\dagger$)   is satisfied. To this end, note that for any $\xi \in D\setm D^p$, there is at most one limit $\alpha\in D^p$ so that  $\xi\in C_\alpha$ (this is the crucial point we use $\textbf C$-closure). So, we can define $q(\xi)=f_\alpha(\xi)$ if such an $\alpha$ exists, and let $q(\xi)=0$ otherwise. In turn, we defined $q\in D_\alpha$ which is an extension of $p$.

(\ref{it:1.5}) is a simple consequence of ($\dagger$), and (\ref{it:2}) follows from the genericity of $G$ and the previous statements.

 \end{cproof}

 Finally, we prove that $\mc P$ is ccc and preserves Suslin trees simultaneously.\footnote{Let us mention that the abstract framework developed in \cite{heikepres} for preserving Suslin trees unfortunately does not fit the uniformization posets.}
 
 \begin{lemma}\label{lm:suslemma}
  $\mc P$ is ccc and for any Suslin tree $R$, $\force_{ \mc P}$ $\check R$ is Suslin.
 \end{lemma}
\begin{cproof} We can assume that $R=(\omg,<_R)$. Suppose that $p\force \dot X\subs R, |\dot X|=\aleph_1$, and $A=\{q_\xi:\xi<\omg\}$ are conditions below $p$. We will find $\xi<\zeta$ and a common extension $q^*$ of $q_\xi$ and $q_\zeta$ so that $q^*\force \dot X$ is not an antichain in $R$.\footnote{The proof that $R$ has no uncountable chains in the extension is completely analogous.}

First, by extending each $q_\xi$, we can suppose that there are $t_\xi\in R$ of height at least $\xi$, so that $q_\xi\force t_\xi\in \dot X$.

 Now, take a continuous sequence of elementary submodels $(M_n)_{n\leq \oo}$ of a large enough $H(\Theta)$ so that $M_0$ contains all the relevant parameters (e.g. $T,\textbf C,\textbf f, \mc P, R,A,(t_\xi)_{\xi<\omg}$). Fix some $q_\xi\in A\setm M_{\oo}$, and find a large enough $n<\oo$ so that $r=q_\xi\cap M_\oo$ is a subset of $M_n$ (and so $r\in M_n$ as well). Let $$E=\bigcup\{C_\alpha\cap M_n:\alpha\in \dom (q_\xi\setm r)\cap \lim \omg\},$$ and note that $E\subs M_n$ is finite, so $E\in M_n$ and $\vareps=\max E\in M_n$ as well.

 Let $D=\ecl{\dom (q_\xi) \cup E}$ and note that $ D\setm (\vareps+1)=\dom (q\setm r)$ by Claim \ref{clm:cl}(\ref{it:maxcl}). We can extend $q_\xi$ to a condition $\bar q_\xi: D\to \oo$ in $ \mc P$, and  we let $\bar r=\bar q_\xi\cap M_n$. Note that $\bar r\in M_n$.

%
%
%

\begin{clm}
 There is $q_\zeta\in M_n\cap A$ and an extension $\bar q_\zeta\in M_n$ with $\dom \bar q_\zeta=D'$ such that 
 \begin{enumerate}
 \item\label{it:one} $\bar q_\xi \cap  \bar q_\zeta= \bar r$,
  \item\label{it:two}  the unique monotone $\psi:D\to D'$ is an isomorphism between $\bar q_\xi$ and $\bar q_\zeta$ which fixes $ D\cap  D'=\dom \bar r$,
  \item\label{it:three} if $\alpha'=\psi(\alpha)$ then $C_\alpha\cap \vareps \sqsubseteq C_{\alpha'}$ and $f_\alpha(\xi)=f_{\alpha'}(\xi)$ for all $\xi\in C_\alpha\cap \vareps$, and
  \item\label{it:four} $t_\zeta\leq_R t_\xi$.
   \end{enumerate}\end{clm}
   \begin{cproof}

   Indeed, $M_n$ contains the set $I\subs \omg$ of all $\zeta\in \omg$ such that $q_\zeta$ has some extension $\bar q_\zeta$ which extends $\bar r$ and is isomorphic to $\bar q_\xi$ in the above sense i.e., conditions (\ref{it:two}) and (\ref{it:three}) are satisfied. Apply Fact \ref{fact:suslin} to the set $Y=\{t_\zeta:\zeta\in I\}$ to find $\zeta\in I\cap M_n$ so that $t_\zeta<_R t_\xi$. In turn, $q_\zeta$ is as desired.

   
   \end{cproof}

   Now, we can amalgamate $\bar q_\xi$ and $\bar q_\zeta$:
   
   \begin{clm}
    $ D\cup  D'$ is $\mathbf C$-closed, and $q^*= \bar q_\xi \cup  \bar q_\zeta$ is a condition in $\mc P$ extending both $q_\xi $ and $q_\zeta$.
   \end{clm}
\begin{cproof}
First, lets see why $ D\cup  D'$ is $\mathbf C$-closed: let $\alpha< \beta\in  D\cup  D'$, and note that the interesting cases are when $\beta \in D\setm M_n$ and  $\alpha \in D'\setm D$. We distiguish two cases: first, assume that $\alpha =\psi(\beta)$. In this case, $C_\beta\cap M_n \subs C_\alpha$, so any $\xi\in C_\alpha\cap C_\beta$ is in $E$ actually.

Second, if  $\alpha\neq \alpha'=\psi(\beta)$ then $C_\alpha\cap C_\beta=C_\alpha\cap C_{\alpha'}$ and the latter is a subset of $D$.

    Finally, the definition of $E$ and the choice of $\bar q_\zeta$ made sure that no element of $\dom q^*\setm \dom \bar q_\xi$ is at critical levels given by $C_\alpha$ for some $\alpha\in  D\setm M_n$. 

\end{cproof}

Hence, $q^*$ is a common extension of both $q_\zeta$ and $q_\xi$, forcing that $\dot X$ is not an antichain.

\end{cproof}

Finally, the theorem is proved.


%
%

 
\end{tproof}

\textbf{Remark.} First, one can also prove that an arbitrary finite support iteration of posets of the form $\mc P_{\textbf f}$ preserve Suslin trees, and so we can show the above consistency result together with arbitrary large values for the continuum. 
\medskip

Second, let us  sketch an argument that the forcing from \cite[Chapter II, Theorem 4.3]{shelah2017proper} (see at the beginning of the section) also preserves Suslin trees. This is a non-essential modification of Lemma \ref{lm:suslemma}. First, recall that each condition $q$ is now defined on a set of the form $\bigcup_{\alpha\in I_q}C_\alpha$ where $I_q$ is finite. Now, suppose that $p\force \dot X$ is an uncountable subset of a Suslin tree $R=(\omg,<_R)$. For each $\xi<\omg$, we take $q_\xi$ below $p$ which forces that $t_\xi\in \dot X$ for some $t_\xi\geq \xi$. Take models $(M_n)_{n\leq \oo}$ as before and fix some $q_\xi\notin M_\oo$. Find $n$ large enough so that $C_\alpha\cap M_\oo\subset M_n$ if $\alpha\in I_{q_\xi}\setm \{M_\oo\cap \omg\}$. Let $r=q_\xi\uhr M_n$ and note that $r\in M_n$. Now, we can look at the set $I\subset \omg$ so that $\zeta\in I$ implies that $r\subset q_\zeta$ and $q_\xi$ and $q_\zeta$ are isomorphic. Apply Fact \ref{fact:suslin} to the set $Y=\{t_\zeta:\zeta\in I\}$ to find $\zeta\in I\cap M_n$ so that $t_\zeta<_R t_\xi$. Note that $q_\zeta\cup q_\xi$  is a common extension of $q_\zeta$ and $q_\xi$ which forces that $\dot X$ is not an antichain.

\section{Uniformization results in ZFC}\label{sec:zfc}

The goal of this section is to present positive uniformization results, provable in ZFC, for some trees of height $\omg$. Our previous results show that we must use trees with uncountable levels. We will mostly work with $\sqt$, the set of all well-ordered, bounded sequences of rational number ordered by end extension. $\sqt$ is a tree of height $\omg$ but contains no uncountable chains (witnessed by the map $t\mapsto \sup_{\mb R}(t)$). Moreover,  $\sqt$ is non-special \cite{stevotrees}.
There is a notable special subtree of $\sqt$, denoted  by $\ssqt$, namely the set of those $t\in \sqt$ which have a maximum. This tree was used by Duro Kurepa to construct the first (special) Aronszajn tree \cite{kunen}.



\medskip

We start by stating our results and then proceed with the proofs.



\begin{thm}\label{thm:Qunif}There is $h:\ssqt\to \oo$ so that for any ladder system $\oo$-colouring $\mathbf f$, there is a (necessarily) special Aronszajn tree $T\subs \ssqt$ so that $h\uhr T$ uniformizes $\mathbf f$.
\end{thm}

\begin{thm}\label{thm:Q*unif}There is  $h:\sqt\to \oo$ so that for any ladder system $\oo$-colouring $\mathbf f$, there is a non-special tree $T\subs \sqt$ so that $h\uhr T$ uniformizes  $\mathbf f$.
\end{thm}

That is, for both trees $\ssqt$ and $\sqt$, there is a single master colouring which witnesses that any ladder system colouring has a $\ssqt$ or $\sqt$-uniformization, respectively. In fact, we isolate a relatively simple combinatorial assumption on functions $h:\sqt\to \oo$ that guarantees $h$ to witness the theorems.

\begin{cor} 
$\unif{\oo}{\ssqt}{\mathbf C}$ and $\unif{\oo}{\sqt}{\mathbf C}$ both hold for any ladder system $\mathbf C$.
\end{cor}

Let us mention that, unlike $\ssqt$ or $\sqt$, no $\aleph_1$-tree can have a single colouring  that uniformizes all ladder system colourings at once.

\begin{prop}
Suppose that $T$ is an $\aleph_1$-tree and $h:T\to \oo$. Then for any ladder system $\mathbf C$, there is a monochromatic 2-colouring of $\mathbf C$ so that $h\uhr S$ is not a  $T$-uniformization of $\mathbf f$ whenever $S\subs T$ is pruned.
\end{prop}


\begin{cproof}
Take any increasing sequence of limit ordinals $(\alpha_n)_{n<\oo}$ with supremum $\alpha_\oo$. List $T_\alpha$ as $(t_n)_{n\in \oo}$ and define $f_{\alpha_n}$ to be constant 0 if and only if there are infinitely many $\xi\in C_{\alpha_n}$ so that $h(t_n\uhr \xi)=1$. One can immediately check that if $h\uhr S$ uniformizes $\mathbf f$ for some downward closed subset $S\subs T$ then $t_n\uhr \alpha_n\notin S$ for all $n<\oo$. In turn,  $S\cap T_{\alpha_\oo}=\emptyset$, a contradiction.
\end{cproof}

Let us present the proofs of the theorems now, starting by some preliminaries.

\medskip

Throughout the rest of this section we only work with subtrees of $\sqt$ and $\ssqt$; hence, we assume the tree to be downward closed and pruned. Now, we say that such a tree $T$ (of possibly countable height) is \emph{strongly pruned} if for any $\vareps<\delta<ht(T)$ and   $s\in T_\vareps$ there are  $t\in T_\delta$ extending $s$ with $|\sup_{\mb R}(t)-\sup_{\mb R}(s)|$ arbitrary small. In case of $T\subs \ssqt$, we could have used $\max_{\mb Q}$ instead of the supremum. For example, the first $\alpha$ levels $(\ssqt)_{<\alpha}$ of $\ssqt$ is a strongly pruned subtree for any $\alpha<\omg$. 


Being strongly pruned allows the extension of a countable tree with an extra level (while preserving being strongly pruned).

\begin{obs}
Suppose that $T\subset \sqt$ is a subtree of countable limit height $\alpha$. Then for any $s\in T$ and $j<\oo$, there is a $t\in \sqt$ of height $\alpha$ so that $s<t$, $t^\downarrow\subset T$ and $|\sup_{\mb R}s-\sup_{\mb R}t|\leq \frac{1}{j}$.
\end{obs}

From this, one readily constructs Aronszajn subtrees of $\sqt$ \cite{kunen}. Now, our goal is to define uniformizations i.e., to inductively extend a tree and a colouring at the same time. This motivates the next definition.

\begin{dfn}
We say that a colouring $h:T\to \oo$ is \textit{rich} if for any $\vareps<\delta<\htt(T)$, $s\in T_\vareps$ and $n\in \oo$ there are  $t\in T_\delta$ extending $s$ with $h(t)=n$ and $|\sup_{\mb R}(t)-\sup_{\mb R}(s)|$ arbitrary small.
\end{dfn} 

In turn, rich maps must be defined on strongly pruned trees. The following observation should make it clear how this definition will be useful.

\begin{obs}
Suppose $T$ is a subtree of $\sqt$ of countable limit height $\alpha<\omg$, $s\in T$ and $h:T\to \oo$ is rich. If $f_\alpha:C_\alpha\to \oo$ for some cofinal, type $\oo$ subset $C_\alpha$ of $\alpha$ then there is a $t\in \sqt$ above $s$  so that
\begin{enumerate}
    \item $t^\downarrow\subs T$, 
    \item $|\sup_{\mb R}(t)-\sup_{\mb R}(s)|$ is arbitrary small and 
    \item $h(t\uhr \xi)=f_\alpha(\xi)$ for almost all $\xi\in C_\alpha$. 
\end{enumerate}
\end{obs}
\begin{cproof}
Indeed, fix some $j<\oo$ and simply define $s<t_0<t_1<\dots $ in $T$ inductively so that for all $k<\oo$,
\begin{enumerate}
    \item $\htt(t_k)=\xi_k$ where $\xi_k$ is the $k$th element of $C_\alpha$ above $\htt{s}$,
\item $h(t_k)=f_\alpha(\xi_k)$, and
\item $|\sup_{\mb R}s-\sup_{\mb R}t_{k}|<\frac{1}{j}$.
    \end{enumerate}
We can always pick the next $t_k$ since $h$ is rich. In the end, $t=\cup_{k<\oo}t_k\in \sqt$ is as desired.
\end{cproof}

 In fact, when proving the theorems, we need a stronger notion. We say that $S$ is an \emph{$\alpha$-bramble} if $S\subset (\sqt)_{<\alpha}$ is a cofinal, rooted, binary subtree and for any branch $b\subset S$, $\cup b$  is a bounded subset of $\sqt$ (and so $\cup b\in \sqt$). 
 
 
 \begin{dfn}
 A map $h:\sqt\to \oo$ is \emph{flush} if it is rich and for any limit $\alpha<\omg$ and $\alpha$-bramble $S$, there is some branch $b\subs S$ so that  $h(t)=n$ for $t=\cup b$. 
 
 

 \end{dfn}
We define a flush map on $\ssqt$ similarly however, as $\ssqt$ branches at limit steps, we need a bit more care.
 \begin{dfn}
 A map $h:\ssqt\to \oo$ is \emph{flush} if it is rich and for any limit $\alpha<\omg$ and $\alpha$-bramble $S$, there is some branch $b\subs S$ so that  $h(t)=n$, $\cup b\subset t$ and $|\sup_{\mb R}\cup b-\sup_{\mb R}t|$ is arbitrary small. 
 
 

 \end{dfn}

We will show that there are flush maps on both $\sqt$ and $\ssqt$, and that any flush map witnesses Theorem \ref{thm:Qunif} and Theorem \ref{thm:Q*unif}.

\begin{lemma}\label{lm:flush}
 There are flush maps on both $\sqt$ and $\ssqt$.
\end{lemma}
\begin{cproof}
We construct flush maps $h_\alpha:(\sqt)_{<\alpha}\to \oo$ for $\alpha<\omg$ by induction, so that $h_\beta$ extends $h_\alpha$ for $\alpha<\beta$. The final union $h=\cup_{\alpha<\omg} h_\alpha$ is the desired map on $\sqt$ then. 

We suppose that $h_\alpha$ has been constructed already and extend it to $\sqt_\alpha$ which, in turn, defines $h_{\alpha+1}$ (in limit steps, we simply take unions). If $\alpha=\alpha_0+1$ is successor then being flush compared to being rich gives no extra requirements. Note that any element $s\in \sqt_{\alpha_0}$ has infinitely many successors $t\in \ssqt_\alpha$ so that $|\sup_{\mb R}(t)-\sup_{\mb R}(s)|$ is arbitrary small. In turn, we can define $h_{\alpha+1}$ on these elements (for each such maximal $s$ independently) so that all colours appear infinitely often on these successors.

Now, suppose that $\alpha$ is limit and list all pairs $(S,n)$ as $(S_\xi,n_\xi)_{\xi<\mf c}$ where $S$ is an $\alpha$-bramble and $n\in \oo$. We inductively pick distinct $b_\xi\subset S_\xi$ which is possible since each $S_\xi$ has continuum many branches. 
Now, simply set $h_{\alpha+1}(t_\xi)=n_\xi$ for $\xi<\mf c$ where $t_\xi=\cup b_\xi$; for all other $t\in \sqt$ of height $\alpha$, we set the value of $h_{\alpha+1}$ arbitrarily. 

Note that we preserved the map being rich. Indeed, for any $s$ of height less than $\alpha$, $n\in \oo$ and $j\in \mb N$, we can build an $\alpha$-bramble $S$ with root $s$ so that $s'\in S$ implies that $|\sup_{\mb R}(s)-\sup_{\mb R}(s')|<\frac{1}{j}$. In turn,  the branch $b$ we picked for the pair $(S,n)$ will give the element $t=\cup b$ that witnesses richness.
This finishes the inductive step and hence the construction of $h$. 




Working on $\ssqt$, almost the same argument gives a flush map again. The only slight difference is that after we selected the appropriate branch $b_\xi$, we pick for each $m\in \mb N$, $t_{\xi,m}\in (\ssqt)_\alpha$ extending $\cup b_\xi$  so that $|\sup_{\mb R}\cup b_\xi-\sup_{\mb R}t_{\xi,m}|<\frac{1}{m}$. Now, we declare $h_{\alpha+1}(t_{\xi,m})=n_\xi$.


\end{cproof}

Finally, we are ready to prove the main theorems.

\begin{tproof}[Proof of Theorem \ref{thm:Qunif}] Our goal is to show that any flush map $h:\ssqt\to \oo$ will work (the existence of which is ensured by Lemma \ref{lm:flush}). 

Fix an $\oo$-colouring $\mathbf f$ of some ladder system $\mathbf C$.  The corresponding tree $T\subs \ssqt$ will be constructed by defining its initial segments $T_{<\alpha}$ for $\alpha<\omg$ by induction on $\omg$ so that

\begin{enumerate}[(a)]
\item $T_{<\alpha}$ is a countable, downward closed and strongly pruned subtree of $\ssqt$ of height $\alpha$,
\item $T_{<\beta}$ is an end extension of $T_{<\alpha}$ for $\alpha<\beta$,\footnote{That is, the trees grow up: $T_{<\alpha}\subset T_{<\beta}$ and any $t\in T_{<\beta}\setm T_{<\alpha}$ has height at least $\alpha$.}
\item $h\uhr T_{<\alpha}$ is rich, and
\item $h\uhr T_{<\alpha}$ uniformizes $\mathbf f\uhr \alpha$.
\end{enumerate}

Clearly, if we succeed then $T=\bigcup_{\alpha<\omg}T_{<\alpha}$ is the tree we were looking for.

\medskip

In limit steps of the induction, we simply take unions. So suppose that $T_{<\alpha}$ is defined and we will find $T_{<\alpha+1}=T_{<\alpha}\cup T_\alpha$ now. If $\alpha$ is a successor ordinal then the extension is straightforward again.

Now, suppose that $\alpha$ is limit. For each $s\in T_{<\alpha}$ and $j<\oo$, we construct an $\alpha$-bramble $S_{s,j}$ with the following properties:
\begin{enumerate}
    \item $S_{s,j}$ has root $s$ and $s'\in S_{s,j}$ implies that $|\sup_{\mb R}(s)-\sup_{\mb R}(s')|<\frac{1}{j}$;
    \item for any $\xi\in C_\alpha$ above $\htt(s)$ and $t\in S_{s,j}$ of height $\xi$, $h(t)=f_{\alpha}(\xi)$.
\end{enumerate}
Why is this possible? We build $S_{s,j}$ as the downward closure of a set $\{t_u:u\in 2^{<\oo}\}\subset T_{<\alpha}$ so that 
\begin{enumerate}[(i)]
    \item $t_\emptyset=s$,
    \item $t_u\subset t_v$ if and only if $u\subset v$,
    \item $|\sup_{\mb R}(s)-\sup_{\mb R}(t_u)|<\frac{1}{j}$ for any $u\in 2^{<\oo}$, and 
    \item\label{p:br} if $u\in 2^n$ then $\htt(t_u)=\xi_n$ and $h(t_u)=f_\alpha(\xi_n)$ where $\xi_n$ is the $n$th element of $C_\alpha$ above  $\htt(s)$.
 \end{enumerate}
 
Given $t_u$, we can easily pick $t_{u\smf 0}$ and  $t_{u\smf 1}$ in the right position using that $h\uhr T_{<\alpha}$ is rich.

Now, since $h$ is flush, for each $n<\oo$ and each pair $s,j$, there is some branch $b_{s,j,n}\subset S_{s,j}$ so that $h(t_{s,j,n})=n$ for an appropriate  $t_{s,j,n}\in (\ssqt)_\alpha$  that extends $\cup b_{s,j,n}$. We simply let $T_\alpha$ collect the countably many elements of the form $t_{s,j,n}$.

First, it is clear that we preserved $h$ being rich on $T_{<\alpha+1}$. Second, for any $t\in T_\alpha$ and for almost any $\xi\in C_\alpha$, $h(t\uhr \xi)=f_\alpha(\xi)$ by property (\ref{p:br}) of the $\alpha$-brambles $S_{s,j}$.  In turn, $h\uhr T_{\alpha+1}$ is still a rich uniformization, as desired.

\medskip

This finishes the inductive construction and hence the proof of the theorem.

\end{tproof}

Finally, we show how to modify the above argument for $\sqt$ and to produce uniformizations on non-special trees $T$ (however, these trees will not be Aronszajn any more).

\begin{tproof}[Proof of Theorem \ref{thm:Q*unif}] Suppose that $h:\sqt\to \oo$ is flush and take any ladder system colouring $\mathbf f$. The corresponding non-special tree $T\subs \sqt$ will be constructed by defining its initial segments $T_{<\alpha}$ for $\alpha<\omg$ by induction on $\omg$ so that

\begin{enumerate}[(a)]
\item $T_{<\alpha}$ is a downward closed and strongly pruned subtree of $\sqt$ of height $\alpha$,
\item $T_{<\beta}$ is an end extension of $T_{<\alpha}$ for $\alpha<\beta$,
\item $h\uhr T_{<\alpha}$ is rich, and
\item $h\uhr T_{<\alpha}$ uniformizes $\mathbf f\uhr \alpha$.
\end{enumerate}
 
 In limit steps, we will simply take unions, and in successor steps $\alpha+1$ where $\alpha$ is also a successor, we only aim to preserve the above properties. This so far is the same as the proof of Theorem \ref{thm:Qunif} except we allowed $T_{<\alpha}$ to have uncountable levels.
 \medskip
 
 Now, the difference (compared to the previous proof) comes  at stages when $T_{<\alpha}$ is already defined for some limit $\alpha$ and we aim to construct the $\alpha$th level $T_\alpha$. Recall that we need to make sure $T$ is non-special i.e., for any partition $g:T\to \oo$ there is some $t'<t\in T$ such that $g(t)=g(t')$.

 First, we need a definition: an \emph{$\alpha$-control tuple} is a 4-tuple $(A,s,g,i)$  that satisfies the following:
\begin{enumerate}[(i)]
 \item $A\subs T_{<\alpha}$ is a countable, downward closed and pruned subtree of height $\alpha$, and $s\in A$,
 \item $h_{<\alpha}\uhr A$ is rich, $g:A\to \oo$ and $i<2$, and
  \item\label{cond:i0} if $i=0$ then there are subtrees $A_\ell\subs A$ for $\ell<\oo$ of height $<\alpha$ so that
  \begin{itemize}
      \item $A=\bigcup_{\ell<\oo} A_\ell$ and $s\in A_0$,
      \item $h\uhr A_\ell$ is rich for all $\ell<\oo$,
      \item for any $n<\oo$, $g^{-1}(n)$ is either dense in each $A_k$ above $s$ (we call $n$ a large colour of $g$ this case),\footnote{I.e., for any $s'\in A_k$ above $s$ there is $s''\in A_k$ above $s'$ so that $g(s'')=n$.} or empty above $s$.
  \end{itemize} 
 
\end{enumerate}

Enumerate all $\alpha$-control tuples $(A,s,g,i)$ extended by an extra pair of natural numbers $1\leq m\in \oo$ and $m'\in \oo$  as $\{(A_\nu,s_\nu,g_\nu,i_\nu,m_\nu,m'_\nu):\nu<\mf c\}$ so that each 6-tuple appears $\mf c$ times. If $i_\nu=0$ then we fix $A_\nu=\bigcup_{\ell<\oo}A_{\nu\ell}$ that witnesses condition (\ref{cond:i0}).
 
 We will define $T_\alpha$ to be $\{t_\nu:\nu<\mf c\}$ so that for any $\nu<\mf c$, 
 \begin{enumerate}[(a)]
 \setcounter{enumi}{4}
     \item\label{it:aa} $s_\nu<t_\nu$ and $t_\nu^\downarrow\subs A_\nu$,
     \item $|\sup_{\mb R} s_\nu-\sup_{\mb R} t_\nu|\leq \frac{1}{m_\nu}$ and $h(t_\nu)=m'_\nu$, 
     \item\label{it:bb} for all $\xi\in C_\alpha$ above the height of $s_\nu$, $h(t_\nu\uhr \xi)=f_\alpha(\xi)$,
      \item \label{it:i} if $i_\nu=0$ then for any $n<\oo$, if $n$ is a large colour of $g_\nu$ then $g_\nu(t_\nu\uhr \vareps)=n$ for some $\vareps<\alpha$, and 
    \item \label{it:ii} if $i=1$ then for any $n<\oo$, there is $\vareps<\alpha$ such that  $g_\nu^{-1}(n)$ is either dense or empty above $t_\nu\uhr \vareps$.
 \end{enumerate}
 If we succeed, then conditions (\ref{it:aa})-(\ref{it:bb}) will ensure that  $h\uhr T_{\alpha+1}$ is a rich uniformization of $\textbf f\uhr \alpha+1$. The last two conditions will help us prove that the final tree is non-special (although, this might not be clear at this point).
 
 \medskip
 
Consider the tuple  $(A_\nu,s_\nu,g_\nu,i_\nu,m_\nu,m'_\nu)$ and let us list $C_\alpha\setm (\htt (s_\nu)+1)$ as $\xi_0<\xi_1<\dots $. We build an $\alpha$-bramble $S^\nu$ as the downward closure of a set $\{t_u:u\in 2^{<\oo}\}\subset T_{<\alpha}$ so that 

\begin{enumerate}
    \item $s_\nu<t_\emptyset$ and $t_u\subset t_v\in A_\nu$ if and only if $u\subs v\in 2^{<\oo}$,
    \item $|\sup_{\mb R}s_\nu-\sup_{\mb R}t_u|<\frac{1}{m_\nu}$ for any $u\in 2^{<\oo}$,
    \item $\htt(t_u)\geq \xi_n$ for all $u\in 2^n$ and $n<\oo$,
    \item if $\xi\in (\htt(t_v)+1)\cap C_\alpha $ then $h(t_v\uhr \xi)=f_\alpha(\xi)$ for any $v\in 2^{<\oo}$,
    \item\label{it:5} if $i=0$ and $n<\oo$ is a large colour of $g_\nu$ then for all $u\in 2^{n+1}$, $g_\nu(t_u)=n$,  and 
    \item\label{it:6} if $i=1$ and $n<\oo$ then for any $u\in 2^{n+1}$, $g_\nu^{-1}(n)$ is either dense or empty above $t_u$.
\end{enumerate}
If we succeed, it should be clear that we indeed get an $\alpha$-bramble $S^\nu$. Moreover, for any branch $b\subset S^\nu$, $|\sup_{\mb R}s_\nu-\sup_{\mb R}\cup b|\leq \frac{1}{m_\nu}$ and $h(\cup b\uhr \xi)=f_\alpha(\xi)$ for all $\xi\in C_\alpha$ above $\htt(s_\nu)$. So, we can apply that $h$ is flush and for each $\nu<\mf c$, we can pick a branch $b_\nu\subset S^\nu$ so that $h(t_\nu)=m'_\nu$ for $t_\nu=\cup b_\nu$. We do this for all $\nu<\mf c$ which defines $T_\alpha=\{t_\nu:\nu<\mf c\}$ with all the required properties. Indeed, conditions (\ref{it:i}) and (\ref{it:ii}) are ensured by (\ref{it:5}) and (\ref{it:6}).



\medskip

We still need to explain the construction of the binary system $\{t_u\}_{u\in 2^{<\oo}}$ corresponding to a fixed 6-tuple $(A_\nu,s_\nu,g_\nu,i_\nu,m_\nu,m'_\nu)$. First, choose $t_\emptyset$  with $\htt(t_\emptyset)=\xi_0$ above $s_\nu$ so that $h(t_\emptyset)=f_\alpha(\xi_0)$. This is possible since $h\uhr A_\nu$ is rich.

Suppose that $(t_u)_{u\in 2^{n}}$ has been defined and fix some $u\in 2^n$. Our goal is to find the right $t_{u\smf 0}$ and $t_{u\smf 1}$.

\begin{clm}
There is a $\bar t\in A_\nu$ above $t_u$ of height at least $\xi_{n+1}$, so that 
\begin{enumerate}
\item if $\xi\in (\htt(\bar t)+1)\cap C_\alpha$ then $h_{<\alpha}(\bar t\uhr \xi)=f_\alpha(\xi)$, and
    \item if $i_\nu=0$ and $\bar t \in A_{\nu \ell}\setm A_{\nu \ell-1}$ then $\htt(\bar t)>\max(C_\alpha\cap \htt(A_{\nu\ell}))$. 
\end{enumerate}

\end{clm}
\begin{cproof}
We look at the elements of $C_\alpha$ between $\htt(t_u)$ and $\xi_{n+1}$, let these be $\xi_{k}<\xi_{k+1}<\dots <\xi_{n+1}$. If, additionally, $i_\nu=0$ then we look at the minimal $\ell<\oo$ so that $\xi_{n+1}<\htt(A_{\nu\ell})$ and $t_u\in A_{\nu\ell}$. We extend our previous list $(\xi_k)_{k\leq n+1}$ by the elements of $C_\alpha\cap \htt(A_{\nu\ell})$ to get $\xi_k<\dots<\xi_K$.

Now, in a finite induction, we define $t_u<\bar t_{k}<\bar t_{k+1}<\dots <\bar t_K=\bar t$ in $A_{\nu \ell}$ (or just in $A_\nu$ if $i=1$) so that $\htt(\bar t_l)=\xi_l$ and $h_{<\alpha}(\bar t_l)=f_\alpha(\xi_l)$. This is possible since $h\uhr A_{\nu\ell}$ is rich.

\end{cproof}

Now, suppose that $\bar t$  is given as above. If $i=0$, $\bar t\in A_{\nu\ell}$ and $n$ is large for $g_\nu$ then we can find $t_{u\smf 0},t_{u\smf 1}$ above $\bar t$ in $A_{\nu\ell}$ so that $g_\nu(t_{u\smf j})=n$ for $j=0,1$. If $i=1$ then we can find $t_{u\smf 0},t_{u\smf 1}$ above $\bar t$ in $A_\nu$ so that  $g_\nu^{-1}(n)$ is either empty or dense above $t_{u\smf j}$ for both $j=0,1$. This defines $t_{u\smf 0},t_{u\smf 1}$ and hence finishes the construction of the binary tree. 

\medskip

At this point, we concluded the construction of the level $T_\alpha$ and so the whole tree $T=\bigcup_{\alpha<\omg}T_\alpha$ is defined. Why is $T$ non special? Suppose that $g:T\to \oo$ and we need to find $t'<t$ so that $g(t')=g(t)$. Take a continuous $\oo+1$-sequence of countable elementary submodels $(M_\ell)_{\ell\leq \oo}$ of $H(\mf c^+)$ with $g,T,h\in M_0$. Let $\alpha=M_\oo\cap \omg$, $A=M_\oo\cap T$ and note that  $h\uhr A$ is rich by elementarity.

\begin{clm}
 There is some $s\in A\cap M_0$ so that $(A,s,g\uhr A,0)$ is an $\alpha$-control tuple.
\end{clm}
\begin{cproof}$(A,\emptyset,g\uhr A,1)$ is certainly an  $\alpha$-control triple so it is enumerated at some step $\mu<\mf c$. The node $t_\mu$ that we introduced (above $\emptyset$) satisfies that for any colour $n<\oo$, $g^{-1}(n)$ is either dense or empty above $t_\mu$ by (\ref{it:ii}). So this property reflects to $M_0$, and we can find an $s\in A\cap M_0$ so that for any colour $n<\oo$,  $g^{-1}(n)$ is either dense or empty above $s$. So,  $(A,s,g\uhr A,1)$ is a control tuple witnessed by the sequence $(A\cap M_\ell)_{\ell<\omega}$.

\end{cproof}

In turn, at step $\alpha$, we enumerated $(A,s,g\uhr A,0)$ as $(A_\nu,s_\nu,g_\nu,i_\nu)$ for some $\nu<\mf c$. First, we claim that $n=g(t_\nu)$ is a large colour of $g_\nu$. Indeed, $t_\nu$ witnesses that  $g_\nu^{-1}(n)$ cannot be empty above any restriction of $t_\nu$. However, in this case there is some $\vareps<\alpha$ so that $g_\nu(t_\nu\uhr \vareps)=n$ by condition (\ref{it:i}). In other words, $g(t_\nu)=g(t_\nu\uhr \vareps)$ as desired.

\end{tproof}

There is another well-investigated class of non-special trees: given a stationary set $I\subset \omg$, let $T_I$ denote the set of all closed subsets of $I$ with the end extension relation \cite{stevocont}. We believe that a non-essential modification of the above arguments yields  the following analogue of Theorem \ref{thm:Q*unif}: there is a single colouring $h^*:T_I\to \oo$ so that any ladder system colouring has a non-special $T_I$-uniformization.

\section{Uniformization results from strong diamonds}\label{sec:diamondunif}

Recall that $\diamondsuit^+$ asserts the existence of a sequence $W=(W_\alpha)_{\alpha<\omg}$ so that $|W_\alpha|\leq \oo$ and for any $X\subs \omg$, there is a closed unbounded $D\subset \omg$ so that  for any $\alpha\in D$,  $$X\cap \alpha,D\cap \alpha\in W_\alpha.$$


%
%
%
%
 
Our main results read as follows.


\begin{thm}\label{thm:dmplus}
 $\diamondsuit^+$ implies that  for any ladder system $\mathbf C$, there is a special Aronszajn tree  $T\subs \ssqt$ so that $\unifc{\oo}{T}{\textbf C}$.
  
\end{thm}

 
 \begin{thm}\label{thm:nonspecial} $\diamondsuit^+$ implies that   for any ladder system $\mathbf C$, there is an Aronszajn tree $T\subs 2^{<\omg}$ so that  any  monochromatic $\oo$-colouring of $\mathbf C$  has a  Suslin $T$-uniformization.\footnote{I.e., the uniformization is defined on a Suslin subtree.}
 \end{thm}
 
 In the latter result, the tree $T$ is necessarily non-special but we don't know if it can be made \emph{almost Suslin} i.e., to contain no stationary antichains. Equivalently, whether $\unifc{\oo}{T}{\textbf C}$ implies the existence of stationary antichains in $T$.

Combined with previous results, we also get the following result.
 
 \begin{cor}\label{cor:cunif}
  $\diamondsuit^+$ implies that for any ladder system $\textbf C$, there is an Aronszajn tree $T$ so that $\unifc{\oo}{T}{\textbf C}$ holds but $\unif{2}{T}{\textbf C}$ fails.
 \end{cor}
 
 We are not aware of an analogue of this result in the classical context of $\omg$-uniformizations. On a related note, it is proved in \cite{whitehead2} that consistently, $\unif{2}{\omg}{\textbf C}$ holds but $\unif{\oo}{\omg}{\textbf C}$ fails for some $\mathbf C$.
 
 \medskip

The idea to prove the above theorems is the following:  in the end, no matter how we constructed the tree $T$, the $\diamondsuit^+$-sequence $(W_\alpha)_{\alpha\in \omg}$ will  allow us to guess initial segments of any ladder system colouring $\mathbf f$ club often. So, at some stage $\alpha$ of constructing the tree (when we constructed $T_{<\alpha}$ already), we look at  elements of $W_\alpha$  that look like a $T_{<\alpha}$-uniformization of some partial ladder system colouring $\mathbf f \uhr \alpha$ that is also in $W_\alpha$; collect these partial uniformizations into a set $\mc H_\alpha$, which must be countable. We aim to extend $T_{<\alpha}$ by the level $T_\alpha$ so that for any element $h$ of $\mc H_\alpha$ and any possible colour $n$ for the ladder $C_\alpha$ there are a lot of $t\in T_\alpha$ so that $t^\downarrow\subs \dom h$ and $h(t\uhr \xi)=n$ for almost all $\xi\in C_\alpha$. This will allow us to extend any map $h\in \mc H_\alpha$ to the level $\alpha$ no matter what is the constant value $n$ of $f_\alpha$.

\medskip

We will use rich maps, strongly pruned trees and some arguments that resemble the work in Section \ref{sec:zfc}. We state one extra preliminary lemma before proving the theorems.

\begin{lemma}\label{lm:unifextend}
 Let $\alpha_0<\alpha_1<\omg$ and suppose that $T$ is a downward closed, strongly pruned subtree of $\ssqt$ of height $\alpha_1$, and $\mathbf f$ is a ladder system $\oo$-colouring. If $h_0$ is a rich $T_{<\alpha_0+1}$-uniformization of $\mathbf f\uhr \alpha_0+1$ then there is a rich $h_1$ extending $h_0$ so that $h_1$ is a  $T_{<\alpha_1}$-uniformization of $\mathbf f\uhr \alpha_1$.

\end{lemma}
\begin{cproof} Look at the poset $\mc P$ of all maps $p$ so that 
\begin{enumerate}
 \item $\dom p=\cup\{t^\downarrow:t\in D^p\}$ where $D_p\subseteq T_{<\alpha_1}$ is finite,
 \item  $p\cup h_0$ is still a function, and
 \item  $p$ uniformizes $\textbf f$: for any $t\in \dom p$ of limit height $\alpha$ and for almost all $\xi\in \dom f_\alpha$, $p(t\uhr \xi)=f_\alpha(\xi)$.  
\end{enumerate}

It is easily checked that for any $p\in \mc P$ and $s\in T_{<\alpha_1}$ extending some element of $\dom p$, there is some $q\in \mc P$ so that $q\supseteq  p$ and $s\in \dom q$. So, one can find a filter $G\subs \mc P$ so that $h_1=h_0\cup(\bigcup G)$ is a rich map on a strongly pruned subtree of $T_{<\alpha_1}$. Indeed, this is done by requiring $G$ to meet an appropriate countable collection of dense subsets of $\mc P$.
\end{cproof}

Finally, given a countable set $W$, we will say that some set $x$ is \emph{$W$-definable} if $x\in M$ whenever $M$ is a countable elementary submodel of $(H(\mathfrak c^+),\in,\prec)$ with $W\in M$, where $\prec$ is some fixed well-order on $H(\mathfrak c^+)$.

\medskip

Let us prove the theorems now.

\begin{tproof}[Proof of Theorem \ref{thm:dmplus}] Let $\textbf W=(W_\alpha)_{\alpha<\omg}$ denote the $\diamondsuit^+$ sequence, and let  $W_{< \beta}=(W _\alpha)_{\alpha< \beta}$ for $\beta<\omg$.

The tree $T\subs \ssqt$ will be constructed by defining its initial segments $T_{<\alpha}$ for $\alpha<\omg$ by induction on $\omg$, along with countable sets $\mc H_\alpha$ and so called \textit{sealing functions} $$\seal_\alpha:\mc H_\alpha\times \oo \to [T_\alpha]^\oo.$$ In the end, these functions will tell us how to continue a partial $T$-uniformization.

We assume that these objects satisfy the following properties, which we preserve throughout the induction:
%
\begin{enumerate}[(a)]
\item $T_{<\alpha}\subs \ssqt$ is a countable, downward closed strongly pruned and uniquely $W_{< \alpha+1}$-definable tree,
\item $T_{<\beta}$ is an end extension of $T_{<\alpha}$ for $\alpha<\beta$,
\end{enumerate}
and for any limit $\alpha<\omg$,
\begin{enumerate}[(a)]
\setcounter{enumi}{2}
\item $\mc H_\alpha$ is the set of all $W_{< \alpha+1}$-definable rich maps $h:S\to \oo$ so that $S\subs T_{<\alpha}$ is strongly pruned,
\item $\seal_\alpha$ is uniquely $W_{< \alpha+1}$-definable, and
 \item\label{seal} for all $h\in \mc H_\alpha$,
 \begin{enumerate}[(i)]
  \item  $\dom (h)\cup \seal_\alpha(h,n)$ is a downward closed, strongly pruned subtree of $T_{\leq \alpha}$, and
  \item $h(t\uhr \xi)=n$ for all $t\in \seal_\alpha(h,n)$ and almost all $\xi\in C_\alpha$.
 \end{enumerate}

\end{enumerate}

These assumptions ensure that for any $h\in \mc H_\alpha$, there is some rich extension $h'$ of $h$ that is defined on   $\dom (h)\cup \seal_\alpha(h,n)$. Moreover, if $h$ was a uniformization for some colouring $\textbf f\uhr \alpha$ of $\textbf C\uhr \alpha$, then $h'$ is a uniformization of $\textbf f\uhr \alpha+1$ whenever $f_\alpha$ is constant $n$.

\medskip

First, the less interesting cases: in limit steps, we simply take unions (which preserves being strongly pruned and the definability requirements). If $\alpha$ is successor and $T_{<\alpha}$ is defined then we canonically extend $T_{<\alpha}$ by a level $T_\alpha$ so that $T_{<\alpha+1}$ remains strongly pruned.

Now, we describe the construction of $T_{<\alpha+1}=T_{<\alpha}\cup T_\alpha$ for a limit $\alpha$ when $T_{<\alpha}$ is defined already. To this end, fix some $h\in \mc H_\alpha$, $s\in \dom h$  and $n,j\in \oo$. Consider the poset $\mc P=\mc P_{s,h,n,j}$ of all finite, non empty chains $p\subs \dom h$ above $s$ so that for all $t\in p$, 
\begin{enumerate}
 \item if $\xi\in C_\alpha$ and $\xi>\htt{s}$ then $h(t\uhr \xi)=n$, and
 \item  $|\max_{\mb Q}(t)-\max_{\mb Q}(s)|<\frac{1}{j}$.
\end{enumerate}

Using that $h$ is rich, it is easy to check  that for any $\vareps<\alpha$, $D_\vareps=\{p\in \mc P:\htt(\cup p)\geq \vareps\}$ is dense in $\mc P_{s,h,n,j}$. Now, take a finite support product $\mc P=\Pi_{i<\oo}\mc P_i$ so that for all $s,h,n,j$ as above, there are infinitely many $i<\oo$ so that $\mc P_i=\mc P_{s,h,n,j}$. It is easy to find a sufficiently generic filter $G\subs \mc P$ so that we get pairwise different, cofinal branches $b_i=\cup G(i)$ in $T_{<\alpha}$ from the $i$th coordinate of $G$. Note that if  $\mc P_i=\mc P_{s,h,n,j}$ then $|\sup_{\mb Q} b_i-\max_{\mb Q} s|\leq \frac{1}{j}$, so we can let $t^*_i=b_i\cup\{\max_{\mb Q}(s)+\frac{1}{j}\}\in \ssqt$ and let $$T_\alpha=\{t_i^*:i<\oo\}.$$

The sealing function $\seal_\alpha$ is defined as follows: given $h\in \mc H_\alpha$ and $n<\oo$, let
$$\seal_\alpha(h,n)=\{t^*_i:i<\oo,\mc P_i=\mc P_{s,h,n,j}\textrm{ for some } j<\oo, s\in \dom h\}.$$

Since the poset $\mc P$ is uniquely definable from $\mc H_\alpha$, we can make a canonical choice of $G$ (using the well order $\prec$), and so $T_\alpha$ and  $\seal_\alpha$ are uniquely definable from $W_{<\alpha+1}$ as well.

\medskip

This finishes the construction of our (special) strongly pruned Aronszajn tree $T=\bigcup_{\alpha<\omg}T_{<\alpha}$. 
\medskip

Suppose now that $\textbf f$ is a monochromatic $\oo$-colouring of $\mathbf C$, and we would like to find a $T$-uniformization. There is a club $D\subs \omg$ so that $\delta\in D$ implies that $\mathbf C\uhr \delta, D\cap \delta,\textbf f\uhr \delta, W_{<\delta}\in W_\delta$. Let $\{\alpha_\xi:\xi<\omg\}$ be the increasing enumeration of $D$, and let  $n_{\xi}$ denote the constant value of $f_{\alpha_\xi}$.

 We will define a $\subseteq$-increasing sequence of functions $(h_\xi)_{\xi<\omg}$ so that

\begin{enumerate}
\item $h_\xi\in \mc H_{\alpha_\xi}$ i.e., $h_\xi$ is a rich $W_{<\alpha_\xi+1}$-definable map on a strongly pruned subtree of $T_{< \alpha_\xi}$, and
\item $h_\xi$ uniformizes $\textbf f\uhr \alpha_\xi+1$.

\end{enumerate}

If we succeed, then $h=\cup_{\xi<\omg}h_\xi$ is the desired $T$-uniformization of $\mathbf f$. 

\medskip

Start by taking the $\prec$-minimal rich map $h_0:T_{<\alpha_0}\to \oo$ which uniformizes $\textbf f\uhr \alpha_0$; note that this is possible by Lemma \ref{lm:unifextend} and that $h_0\in \mc H_{\alpha_0}$ as $\mathbf{f}\uhr \alpha_0\in W_{\alpha_0}$.

Now, suppose we constructed some $h_{\xi}\in \mc H_{\alpha_\xi}$. In turn, $\seal_{\alpha_\xi}(h_\xi,n_\xi)$ is defined (where $n_\xi$ is the constant value of $f_{\alpha_\xi}$). So, we can take a $\prec$-minimal map $$\bar h_\xi:\dom h_\xi\cup \seal_{\alpha_\xi}(h_\xi,n_\xi)\to \oo$$ that is a rich extension of $h_\xi$. By the properties of the sealing function, $\bar h_\xi$ remains a uniformization of $\mathbf f\uhr \alpha_\xi+1$. Now, take the $\prec$-minimal rich extension $h_{\xi+1}$ of $\bar h_\xi$ that uniformizes $\mathbf f\uhr \alpha_{\xi+1}$; this exists by Lemma \ref{lm:unifextend}. The minimal choices and that $\mathbf f\uhr \alpha_{\xi+1}\in W_{\alpha_{\xi+1}}$ implies that $h_{\xi+1}\in \mc H_{\alpha_{\xi+1}}$.




In limit steps $\xi$, we simply take unions: $ h_\xi=\cup_{\zeta<\xi} h_\zeta$. This map is uniquely definable from the parameters $\mathbf C\uhr \delta, D\cap \delta,\textbf f\uhr \delta, W_{<\delta}$ (due to the canonical choices we made along the way), and these parameters are all in $W_\delta$. Hence $h_\xi$ is $W_{<\delta+1}$-definable as desired i.e., $h_\xi\in \mc H_{\alpha_\xi}$.
\medskip

This finishes the induction and hence the construction of the $T$-uniformization $h$. 

\medskip

To summarize, we constructed an Aronszajn tree $T\subseteq \ssqt$ that satisfies $\unifc{\oo}{T}{\mathbf C}$.


%

\end{tproof}

 Our next goal is to prove the second theorem about Suslin uniformizations, which will be quite similar. We will construct $T$ as a subtree of $2^{<\omg}$ instead of $\ssqt$, and make sure that the sealing functions preserve maximal antichains that are guessed by the diamond sequence. 
 
 We need a small definition before the proof: suppose that $S$ is a countable tree of height $\alpha$ and $X\subs S$ is a maximal antichain. We say that $X$ is \emph{$\alpha$-reflecting} if the set $$\{\vareps\in \alpha\cap \lim(\omg):X\cap S_{<\vareps} \textmd{ is a maximal antichain in }S_{<\vareps}\}$$ is cofinal in $\alpha$.
 
\begin{tproof}[Proof of Theorem \ref{thm:nonspecial}] We still use $\textbf W$ to denote the $\diamondsuit^+$-sequence. $T$ will be built by constructing its initial segments $T_{<\alpha}\subs 2^{<\omg}$ along with $\mc H_\alpha$ and $\seal_\alpha$ with the following properties:

\begin{enumerate}[(a)]
\item $T_{<\alpha}\subs 2^{<\omg}$ is a countable, downward closed, normal tree\footnote{I.e., any node has incomparable extensions.} which is uniquely definable from $W_{< \alpha+1}$,
\item $T_{<\beta}$ is an end extension of $T_{<\alpha}$ for $\alpha<\beta$,
\end{enumerate}
and for any limit $\alpha<\omg$,
\begin{enumerate}[(a)]
\setcounter{enumi}{2}
\item $\mc H_\alpha$ is the set of all $W_{< \alpha+1}$-definable rich maps $h:S\to \oo$ so that $S\subs T_{<\alpha}$ is downward closed and normal,
\item $\seal_\alpha$ is uniquely definable from $W_{< \alpha+1}$, and
 \item for all $h\in \mc H_\alpha$,
 \begin{enumerate}[(i)]
  \item  $\dom (h)\cup \seal_\alpha(h,n)$ is a downward closed subtree of $T_{\leq \alpha}$, 
  \item $h(t\uhr \xi)=n$ for all $t\in \seal_\alpha(h,n)$ and almost all $\xi\in C_\alpha$, and
  \item if $X\in W_\alpha$ is an $\alpha$-reflecting maximal antichain in $\dom h$ then any $t\in \seal_\alpha(h,n)$ is above some element of $X$.
 \end{enumerate}
 
 \end{enumerate}



As before, the interesting case in the construction is when $T_{<\alpha}$ is constructed already for a limit $\alpha$ and we aim to define $T_\alpha$. For any $h\in \mc H_\alpha$, $s\in \dom h$ and $n\in \oo$, we now consider the poset $\mc P_{s,h,n}$ of all finite, non empty chains $p\subs \dom h$ above $s$ so that for all $t\in p^\downarrow$ above $s$,  if $\xi\in C_\alpha$ and $\xi>\htt{s}$ then $h(t\uhr \xi)=n$.

We take a finite support product $\mc P=\Pi_{i<\oo}\mc P_i$ so that for all $s,h,n$ as above, there are infinitely many $i<\oo$ so that $\mc P_i=\mc P_{s,h,n}$. In order to find the right generic branches, we need a new density lemma.

\begin{clm}
 Suppose that $X\subs \dom h$ is an $\alpha$-reflecting maximal antichain. Then $$D_X=\{p\in \mc P_{s,h,n}:p^\downarrow\cap X\neq \emptyset\}$$ is dense in $\mc P_{s,h,n}$.
\end{clm}
\begin{cproof}Fix some $\alpha$-reflecting maximal antichain $X\subs S=\dom h$ and suppose that $p_0\in P_{s,h,n}$ is arbitrary. Let $t_0=\max_S(p_0)$ and find some limit $\vareps<\alpha$ above $t_0$ so that $X\cap S_{<\vareps}$ is a maximal antichain in $S_{<\vareps}$. List the elements of $C_\alpha\cap (\vareps\setminus \htt{(t_0)})$ as $\alpha_1<\alpha_2<\dots<\alpha_{k-1}$. Using that $h$ is rich, we can find $t_0<t_1<t_2<\dots<t_{k-1}$ so that $t_i\in S_{\alpha_i}$ and $h(t_i)=n$ for $1\leq i<k$. In turn, $p_1=p_0\cup\{t_i:i<k\}\in P_{s,h,n}$ is an extension of $p_0$.

Now,  either $X\cap t_{k-1}^\downarrow\neq \emptyset$ already (which means $p_1\in D_X$) or we can find some $t\in S_{<\vareps}$ above $t_{k-1}$ so that $t^\downarrow \cap X\neq \emptyset$. Note that no new elements of $C_\alpha$ appear between the heights of  $t_{k-1}$ and $t$, so we can define $p_2=p_1\cup \{t\}$ which is an extension of  $p_0$ in $D_X$.  
 
\end{cproof}

 
 Now, it is easy to find a sufficiently generic filter $G\subs \mc P$ so that we get pairwise different branches $b_i=\cup G(i)$ in $T_{<\alpha}$ from the $i$th coordinate of $G$ and $b_i^\downarrow \cap X\neq \emptyset$ for any $i<\oo$ and $\alpha$-reflecting maximal antichain $X\subs \dom h$ so that $X\in W_\alpha$ (where $P_i=P_{s,h,n}$). 

We define $T_\alpha$ by adding unique upper bounds $t^*_i$ for each branch $b_i$, just as in the proof of Theorem \ref{thm:dmplus}. The sealing function $\seal_\alpha$ is defined as follows: given $h\in \mc H_\alpha$ and $n<\oo$, let
$$\seal_\alpha(h,n)=\{t^*_i:i<\oo,\mc P_i=\mc P_{s,h,n},s\in \dom h\}.$$

Since the poset $\mc P$ is uniquely definable from $\mc H_\alpha$, we can make a canonical choice of $G$, and so $T_\alpha$ and $\seal_\alpha$ are $W_{<\alpha+1}$-definable as well. 

\medskip

This finishes the construction of an Aronszajn tree $T=\bigcup_{\alpha<\omg}T_{<\alpha}\subs 2^\omg$.

\medskip

Suppose that we are given some monochromatic $\oo$-colouring $\textbf f$ of $\mathbf C$. The diamond sequence and sealing functions provide a unique way to construct a sequence of maps $(h_\xi)_{\xi<\omg}$ along a club $D=\{\alpha_\xi:\xi<\omg\}\subs \omg$ so that

\begin{enumerate}
\item $h_\xi\in \mc H_{\alpha_\xi}$, 
\item $(\dom h_{\xi+1})\cap T_{\alpha_\xi}=\seal_{\alpha_\xi}(h_\xi,n_\xi)$  where $n_\xi$ is the constant value of $f_{\alpha_\xi}$, and
\item $h_\xi$ uniformizes $\textbf f\uhr \alpha_\xi+1$.
\end{enumerate}

So $h=\bigcup_{\xi<\omg}h_\xi$ is a $T$-uniformization of $\mathbf f$. We need to show that $\dom h=S$ is Suslin; this amounts to proving that if $X\subs S$ is a maximal antichain then $X$ is countable.  The set $\tilde D$ of those $\alpha<\omg$ so that $X\cap S_{<\alpha}\in W_\alpha$ and $X\cap S_{<\alpha}$ is an $\alpha$-reflecting maximal antichain in $S_{<\alpha}$ contains a club. So, we can find some $\xi<\omg$ so that $\alpha_\xi\in \tilde D$. In turn, by the properties of our sealing function, any $t\in \seal_{\alpha_\xi}(h_\xi,n_\xi)=S_{\alpha_\xi}$ extends some element of $X\cap S_{<\alpha_\xi}$. So we must have $X=X\cap S_{<\alpha_\xi}$ and hence $X$ is countable.

\medskip
This finishes the proof of the theorem.

\end{tproof}

 \section{Uniformization and large antichains}\label{sec:largeac}
 
 First, let us explain why  large uncountable antichains appear even in the non-special tree $T$ of Theorem \ref{thm:nonspecial}. Fix a ladder system $\mathbf C$, and consider the colourings $\mathbf f^\nu$ for $\nu\in \lim \omg$ so that $f^\nu_\alpha$ is constant 1 on $C_\alpha$ if $\nu\leq \alpha<\omg$, otherwise $f^\nu_\alpha$ is constant 0. So each $\mathbf f^\nu$ is a monochromatic 2-colouring of $\mathbf C$.
 
 We can arrange the $\diamondsuit^+$-sequence $(W_\alpha)_{\alpha<\omg}$ so that $\lim(\omg)\cap \alpha,\mathbf f^\nu\uhr \alpha\in W_\alpha$ for $\nu\leq \alpha$. So, as in the proof of Theorem \ref{thm:nonspecial}, we can build all the uniformizations $h^\nu$ for the colouring $\mathbf f^\nu$  using the same club $D=\lim \omg$. 
 
 Let $S^\nu\subs T$ be the domain of $h^\nu$. Since $\mathbf f^\mu\uhr \mu=\mathbf f^\nu \uhr \mu$ for $\mu<\nu$,  we must have $(S^\mu)_{<\mu}=(S^\nu)_{<\mu}$ and $h^\nu$ agrees with $h^\mu$ here. Now, $f^\mu_\mu=1\neq 0=f^\nu_\mu$ implies that $$(S^\mu)_{\mu}\cap (S^\nu)_{\mu}=\emptyset$$ whenever $\mu<\nu<\omg$. In turn, any selection $s_\mu\in (S^\mu)_\mu$ defines an uncountable antichain $\{s_\mu:\mu\in \lim \omg\}$ of $T$ meeting a club set of levels.

\medskip

The antichain above appeared because the map that assigned the uniformization to the colouring was continuous: if two colourings agreed up to some level, their corrensponding uniformizations agreed up to that level as well. However, the fact alone that the above defined colourings all have $T$-uniformizations does not necessarily mean that $T$ has large antichains.

\begin{thm}\label{thm:family}Suppose that $\mc F$ is a family of $\aleph_1$ many ladder system colourings.\footnote{The ladder systems may vary with the colouring.} Then
if $\diamondsuit$ holds then  there is a Suslin tree $T$ so that any $\textbf f\in \mc F$ has a full $T$-uniformization.
\end{thm}

One can similarly show, without any extra assumptions beyond ZFC, that for any family $\mc F$ of $\aleph_1$ many ladder system colourings, there is a (necessarily special) Aronszajn tree $T\subseteq \ssqt$ so that any $\textbf f\in \mc F$ has a full $T$-uniformization.

%

\begin{tproof}
Let $\mc F=\{\textbf f^\nu:\nu<\omg\}$ where $\mathbf f^\nu=(f^\nu_\alpha)_{\alpha\in \lim(\omg)}$ and $\dom f^\nu_\alpha=C^\nu_\alpha$. Let $\mathbf W$ denote the diamond sequence.

We aim to construct a Suslin tree $T\subs 2^{<\omg}$ and the uniformizations $h^\nu:T\to \oo$ for $\textbf f^\nu$  simultaneously for each $\nu<\omg$. So, by induction on $\alpha<\omg$, we construct $T_{<\alpha}$ and $(h^\nu_ \alpha)_{\nu< \alpha}$ so that 
\begin{enumerate}
\item $T_{<\alpha}$ is a countable, downward closed and normal subtree of $2^{<\omg}$,
 \item $h^\nu_ \alpha:T_{<\alpha}\to \oo$ is a uniformization of $f^\nu\uhr \alpha$,
 \item for all $\nu< \alpha<\beta$, $T_{<\beta}$ is an end-extension of $T_{<\alpha}$ and $h^\nu_\beta$ is an end-extension of $h^\nu_ \alpha$,
 \item\label{it:sus} if $\alpha$ is limit and $W_\alpha$ is a maximal antichain of $T_{<\alpha}$ then any $t\in T_\alpha$ extends some element of $W_\alpha$,
 \item\label{it:simul} for any $\Delta\in [\alpha]^{<\oo}$, $n:\Delta\to \oo$,  $\vareps<\alpha$ and $s\in T_{<\vareps}$ there are infinitely many $t\in T_\vareps$ above $s$ so that $$h^\nu(t)=n(\nu)$$ for all $\nu\in \Delta$.
\end{enumerate}

The latter condition is a simultaneous version of the richness condition that we introduced before the proof of Theorem \ref{thm:dmplus}. It is clear that if we succeed in building these objects then $T=\bigcup_{\alpha<\omg}T_{<\alpha}$ is a Suslin tree and $h^\nu=\bigcup_{\nu< \alpha<\omg}h^\nu_\alpha$ is a full $T$-uniformization of $\mathbf f^\nu$.


 
 \medskip
 
 
 In limit steps, we can take unions and all the assumptions are preserved.  As usual, the non-trivial step in our construction is when we are given $T_{<\alpha}$ for some limit $\alpha$ along with $h^\nu_\alpha:T_{<\alpha}\to \oo$ for $\nu<\alpha$ with the above conditions. We need to define $T_\alpha$, which is the next level of the tree, extend all the maps $h^\nu_\alpha$ to $T_\alpha$ (which gives $h^\nu_{\alpha+1}$), and define $h^\alpha_{\alpha+1}$ all while preserving condition (\ref{it:simul}).
 
 We first construct a map $h^\alpha_\alpha:T_{<\alpha}\to \oo$ that will serve as the restriction of $h^\alpha_{\alpha+1}$ to $T_{<\alpha}$.
 
 \begin{clm}
There is a map $h^\alpha_\alpha:T_{<\alpha}\to \oo$ that uniformizes $\textbf f^\alpha\uhr \alpha$ and so that for any $\Delta\in [\alpha+1]^{<\oo}$, $n:\Delta\to \oo$,  $\vareps<\alpha$ and $s\in T_{<\vareps}$, there are infinitely many $t\in T_\vareps$ above $s$ so that $$h^\nu(t)=n(\nu)$$ for all $\nu\in \Delta$. 
 \end{clm}
\begin{cproof}
 It is straightforward to check that the forcing $\mc P$ defined in the proof of Lemma \ref{lm:unifextend} can be used to provide this map $h^\alpha_\alpha\uhr T_{<\alpha}$, by choosing the filter $G\subs \mc P$ to be generic enough over the maps $\{h^\nu_\alpha:\nu<\alpha\}$.
\end{cproof}

 Our next goal is to define $T_\alpha$: first, list $T_{<\alpha}$ as $\{s_i:i<\oo\}$, each node infinitely often. We will inductively construct cofinal branches $b_i\subs T_{<\alpha}$ above $s_i$ so that 
 \begin{enumerate}[(a)]
  \item $b_i\neq b_j$ for all $j<i$,
  \item $b_i$ extends some element of $W_\alpha$ if $W_\alpha$ is a maximal antichain in $T_{<\alpha}$, and
  \item for any $\nu\leq \alpha$, and almost all $\xi\in C^\nu_\alpha $, $$h^\nu_\alpha(\cup b_i\uhr \xi)=f^\nu_\alpha(\xi).$$
 \end{enumerate}
If we succeed then we let $t^*_i=\cup b_i$ and $T_\alpha=\{t^*_i:i<\oo\}$. The last condition ensures that for any $\nu\leq \alpha$, $h^\nu_\alpha$ uniformizes $\mathbf f^\nu\uhr \alpha+1$ as well and not just $\mathbf f^\nu\uhr \alpha$.

Suppose that $b_j$ is already defined for $j<i$. Let $t_0\in T_{<\alpha}$ extend $s_i$ so that $t_0\notin \cup_{j<i} b_j^\downarrow$ and $t_0$ extends some element of $W_\alpha$ if  $W_\alpha$ is a maximal antichain in $T_{<\alpha}$. List $\alpha+1$ as $\{\nu_n:n<\oo\}$.  We now build a sequence $t_0<t_1<\dots<t_n<\dots $ for $n<\oo$ in $T_{<\alpha}$ so that
\begin{enumerate}[(i)]
 \item $\alpha_n=\htt(t_n)$ is the $n$th element of $C_\alpha^{\nu_0}$,\footnote{The only role of this assumption is to make sure that $\sup_{n\in \oo}\htt(t_n)=\alpha$.}
 \item if $\xi\in (\alpha_{n+1}+1)\setm (\alpha_n+1)$ then for any $k\leq n$ such that $\xi\in C_\alpha^{\nu_k}$,  $$h^{\nu_k}_\alpha(t_{n+1}\uhr \xi)=f^{\nu_k}_\alpha(\xi).$$
\end{enumerate}
Clearly, the branch $b_i=\cup_{n<\oo}t_n^\downarrow$ will satisfy our requirements.

\medskip

 Given $t_n$, we describe the construction of $t_{n+1}$. Consider the finite set $$\bigl (\bigcup\{ C_\alpha^{\nu_k}:k\leq n\}\bigr )\cap(\alpha_{n+1}+1)\setm (\alpha_n+1),$$ and let $\{\xi_l:1\leq l<\ell\}$ denote the increasing enumeration. We need to find an increasing sequence in $T_{<\alpha}$ $$t_n=t_{n,0}<t_{n,1}<\dots<t_{n,l}<\dots<t_{n,\ell-1}=t_{n+1}$$  so that $\htt(t_{n,l})=\xi_l$ and if $\xi_l\in C^{\nu_k}_\alpha$ for some $k\leq n$ then $h^{\nu_k}_\alpha(t_{n,l})=f^{\nu_k}_\alpha(\xi_l)$ for $1\leq l<\ell$. Given $t_{n,l}$, we simply apply the simultaneous richness property (\ref{it:simul}) of the colourings $(h^{\nu_k}_\alpha)_{k\leq n}$ to find $t_{n,l+1}$ (this is done by setting $\vareps=\xi_{l+1}$, $\Delta=\{\nu_k:k\leq n, \xi_{l+1}\in C^{\nu_k}_\alpha\}$ and $n(\nu_k)=f^{\nu_k}_\alpha(\xi_{l+1})$).

 \medskip
 
 Now that we defined these branches and the level $T_\alpha$, our last job is to define the maps $h^\nu_{\alpha+1}$ (extending $h^\nu_\alpha$) on the new level $T_\alpha$ for $\nu\leq \alpha$. The only requirement to keep in mind is the simultaneous richness condition (\ref{it:simul}) with $\vareps=\alpha$. One can do this by a simple induction of length $\oo$ enumerating all these countably many requirements, or by using the next claim.
 
 \begin{clm}
  Suppose that $M$ is a countable elementary submodel so that $T_{\leq \alpha}\in M$, and let $(c^\nu)_{\nu\leq \alpha}$ be mutually $M$-generic Cohen-functions from $T_\alpha$ to $\oo$. If we let $h^\nu_{\alpha+1}(t)=c^\nu(t)$ for $t\in T_\alpha$ then  $(h^\nu_{\alpha+1})_{\nu\leq \alpha}$ satisfies condition (\ref{it:simul}). 
 \end{clm}
 
 In any case, we leave the proof to the reader.
 
 \medskip
 
 With this, the $\alpha$th step of the induction is done and hence we finished our contruction of the tree $T$ with the full $T$-uniformizations. The fact that $T$ is Suslin follows form condition (\ref{it:sus}).

\end{tproof}

\section{Closing remarks and open problems}\label{sec:problems}

Hopefully, our exposition convinced the reader that there is a diverse theory behind Definition \ref{dfn:Tunif} well worth studying in detail. At the same time, we believe that our proofs demonstrated rather different  applications of a wide spectrum of guessing principles, starting from the weakest of weak diamonds ($2^{\aleph_0}<2^{\aleph_1}$), and length continuum diagonalization arguments in ZFC, to one of the strongest assumptions, $\diamondsuit^+$.
Finally, let us make a few remarks on Kurepa trees and emphasize the still unsolved questions that occured throughout our manuscript.

%


\subsection*{Remarks on Kurepa trees}

A \textit{Kurepa tree} is an $\aleph_1$-tree with at least $\aleph_2$ branches. 
One can construct Kurepa trees $T$ from $V=L$ or using strong diamond assumptions \cite{jech,devlin1982combinatorial}. Moreover, it is also possible to exclude Aronszajn subtrees in these construction. Let us also mention a rather flexible forcing approach due to Todorcevic \cite{MR631563} to achieve the same result (with the extra feature that complete binary subtrees are also avoided). Finally, we mention \cite[Lemma 5.8]{MR631563} where it is proved that any ccc forcing preserves that an $\aleph_1$-tree contains no Aronszajn subtrees.

Now, what can we say about $T$-uniformizations for a Kurepa tree $T$? First, let us look at models of CH or the weak diamond $\Phi^2_\omg$. Now, for any ladder system $\mathbf C$, there is a colouring $\mathbf f$ without an $\omg$-uniformization. So, if $\unif{\oo}{T}{\mathbf C}$ holds then any $T$-uniformization $\varphi$ of this particular $\mathbf f$ must be defined on an Aronszajn subtree of $T$. In particular, $T$ must have Aronszajn subtrees. Of course, we can make $\unif{\oo}{T}{\mathbf C}$ hold for a Kurepa tree 'artificially' by requiring that $T$ has an Aronszajn subtree $T_0$ that satisfies  $\unif{\oo}{T_0}{\mathbf C}$. We are not sure if this is necessarily the case.

\begin{prob} Is it consistent with CH that there is a Kurepa tree $T$ with $\unif{\oo}{T}{\mathbf C}$ so that for any Aronszajn subtree $T_0$ of $T$, $\unif{\oo}{T_0}{\mathbf C}$ fails.
\end{prob}

Now, if $\unif{\oo}{T}{\mathbf C}$ and $T$ has no Aronszajn subtrees then actually $\unif{\oo}{\omg}{\mathbf C}$ holds (since the domain of any $T$-uniformization includes an $\aleph_1$-branch). In particular, CH must fail. Note that the above cited results from \cite{MR631563} imply that even $MA_{\aleph_1}$ (and so $\unif{\oo}{\omg}{\textbf C}$ for all $\textbf C$) is consistent with the existence of Kurepa trees with no Aronszajn subtrees.
 \medskip

\subsection*{Various problems} 
First, it is natural to ask how crucial it was in Theorem \ref{thm:nonunifs} that all antichains are countable.

\begin{prob} Suppose CH holds. Does $\unifc{\oo}{T}{\textbf C}$ for some $\mathbf C$ imply the existence of stationary antichains in $T$? 

 \end{prob}


Upon reading Theorem \ref{thm:family}, it is also natural to consider the smallest size $\lambda^*$ of a family $\mc F$ of ladder system colourings such that there is no single Aronszajn $T$ such that any $\textbf f\in \mc F$ has a full $T$-uniformization. Note that under $\textmd{MA}_{\aleph_1}$, no such family exists since all ladder system colourings have an $\omg$-uniformization. However, if  $\lambda^*$ exists then it is at least $\aleph_2$ by (the remark after) Theorem \ref{thm:family}. 

\begin{prob} 
 Is it consistent that $\lambda^*$ exists and is bigger than $\aleph_2$?
\end{prob}

\medskip

In Section \ref{sec:forceunif}, we cited \cite[Theorem 6.2]{larson2001chain} stating that a Suslin tree always forces $\neg \unifc{2}{\omg}{\mathbf C}$. It seems unclear if that argument can be extended to show the existence of non $T$-uniformizable colourings.


\begin{prob}
Suppose $S$ is a Suslin tree and $\mathbf C$ and $T$ are ($S$-names for) a ladder system and Aronszajn tree. Does $S$ force $\neg \unif{2}{T}{\mathbf C}$ or even $\neg \unifc{2}{T}{\mathbf C}$?
 \end{prob}

In \cite{soukup2018model}, we presented a model of CH so that $\unif{\oo}{T}{\textbf C}$ holds for some Aronszajn trees $T$ but fails for others (e.g., for any Suslin tree in that model). It would be nice to solve the following.

\begin{prob}
Given two sufficiently different\footnote{One should possibly require no club-embedding of $T$ into $S$.} trees $S$ and $T$, can we force $\unif{\oo}{T}{\textbf C}$ together with the failure of $\unif{\oo}{S}{\textbf C}$?
\end{prob}

Another natural question is whether the number of colours plays a crucial role in the unformization property for trees.

\begin{prob}
Suppose that $T$ is an Aronszajn tree and $\textbf C$ is a ladder system. Does $\unif{2}{T}{\textbf C}$ imply $\unif{\oo}{T}{\textbf C}$?
\end{prob}

In the classical setting of $\omega_1$-uniformizations, the answer is no: in fact, if a ladder system $\textbf C=(C_\alpha)_{\alpha\in I}$ is only defined on a stationary, co-stationary set $I\subs \omg$ then the conjunction of CH plus $\unif{2}{\omg}{\textbf C}$ and $\neg \unif{\oo}{\omg}{\textbf C}$ are consistent  \cite{whitehead2}. 

The following question is inspired by Corollary \ref{cor:cunif} and concerns the classical $\omg$-uniformization theory.

\begin{prob}
Is it consistent that $\unifc{\oo}{\omg}{\textbf C}$ holds but $\unif{2}{\omg}{\textbf C}$ fails for some ladder system $\mathbf C$?
\end{prob}

 We conjecture that the answer is yes, and one might start by looking at the \cite{shelah1989consistent} and \cite{whitehead2} where similar questions are studied.

%

\medskip

Let us point out that we focused solely on ladder systems ranging over all countable limit ordinals. It would be very interesting to see how the tree uniformization theory changes if one restricts the ladder system to a stationary, co-stationary set. The above cited result shows that this distinction is crucial in the classical context.


\begin{prob}
Suppose that $\textbf{C}$ is a stationary, co-stationary ladder system. Does CH imply that whenever $T$ is a Suslin tree then there is a (monochromatic) 2-colouring of $\textbf{C}$ with no $T$-uniformization?
\end{prob}

Finally, it would be interesting to look at ladder systems and trees on $\aleph_2$, and to study uniformizations there. This was initiated  in \cite[Appendix 3]{shelah2017proper} for the classical theory. Whether the uniformization property here also gives consequences on minimal linear orders of size $\aleph_2$ is an interesting question  for future research.

\bibliographystyle{plain}
\bibliography{thesis}

\begin{thebibliography}{10}

\bibitem{avraham1978consistency}
Uri Abraham, K~J Devlin, and Saharon Shelah.
\newblock The consistency with {CH} of some consequences of {M}artin's axiom
  plus $2^{\aleph_0}>{\aleph_1}$.
\newblock {\em Israel Journal of Mathematics}, 31(1):19--33, 1978.

\bibitem{avraham1982forcing}
Uri Abraham and Saharon Shelah.
\newblock Forcing with stable posets.
\newblock {\em The Journal of Symbolic Logic}, 47(1):37--42, 1982.

\bibitem{abrahamiso}
Uri Abraham and Saharon Shelah.
\newblock Isomorphism types of {A}ronszajn trees.
\newblock {\em Israel Journal of Mathematics}, 50(1):75--113, 1985.

\bibitem{balogh2004uniformization}
Zolt{\'a}n Balogh, Todd Eisworth, Gary Gruenhage, Oleg Pavlov, and Paul
  Szeptycki.
\newblock Uniformization and anti-uniformization properties of ladder systems.
\newblock {\em Fund. Math}, 181:189--213, 2004.

\bibitem{baumorder}
James~E Baumgartner.
\newblock Order types of real numbers and other uncountable orderings.
\newblock In {\em Ordered sets}, pages 239--277. Springer, 1982.

\bibitem{devlin1982combinatorial}
Keith~J Devlin.
\newblock The combinatorial principle $\diamondsuit^\sharp$.
\newblock {\em Journal of Symbolic Logic}, pages 888--899, 1982.

\bibitem{devlin1978weak}
Keith~J Devlin and Saharon Shelah.
\newblock A weak version of $\diamondsuit$ which follows from
  $2^{\aleph_0}<2^{\aleph_1}$.
\newblock {\em Israel Journal of Mathematics}, 29(2-3):239--247, 1978.

\bibitem{devlin1979note}
Keith~J Devlin and Saharon Shelah.
\newblock A note on the normal moore space conjecture.
\newblock {\em Canad. J. Math}, 31:241--251, 1979.

\bibitem{eklof1992uniformization}
Paul~C Eklof, Alan~H Mekler, and Saharon Shelah.
\newblock Uniformization and the diversity of {W}hitehead groups.
\newblock {\em Israel Journal of Mathematics}, 80(3):301--321, 1992.

\bibitem{eklof1994hereditarily}
Paul~C Eklof, Alan~H Mekler, and Saharon Shelah.
\newblock Hereditarily separable groups and monochromatic uniformization.
\newblock {\em Israel Journal of Mathematics}, 88(1-3):213--235, 1994.

\bibitem{osvaldo}
Osvaldo Guzman and Michael Hrus\'ak.
\newblock Parametrized $\diamondsuit$-principles and canonical models.
\newblock {\em Slides from Retrospective workshop on Forcing and its
  applications, Fields Institute}, 2015.

\bibitem{jech}
Thomas Jech.
\newblock Set theory. the third millennium edition.
\newblock {\em Springer Monographs in Mathematics. Springer-Verlag}, 2003.

\bibitem{kunen}
K.~Kunen.
\newblock {\em Set {T}heory, an introduction to independence proofs}.
\newblock Elsevier, 2014.

\bibitem{larson2016automorphisms}
Paul Larson and Paul McKenney.
\newblock Automorphisms of ${P}(\lambda)/{I}_\kappa$.
\newblock {\em Fund. Math}, 233:271--291, 2016.

\bibitem{larson2001chain}
Paul Larson and Stevo Todor{\v{c}}evi{\'c}.
\newblock Chain conditions in maximal models.
\newblock {\em Fundamenta Mathematicae}, 1(168):77--104, 2001.

\bibitem{larson2002katetov}
Paul Larson and Stevo Todorcevic.
\newblock Katetov's problem.
\newblock {\em Transactions of the American Mathematical Society},
  354(5):1783--1791, 2002.

\bibitem{heikepres}
Heike Mildenberger and Saharon Shelah.
\newblock Many countable support iterations of proper forcings preserve
  {S}ouslin trees.
\newblock {\em Annals of Pure and Applied Logic}, 165(2):573--608, 2014.

\bibitem{minami2008suslin}
Hiroaki Minami.
\newblock Suslin forcing and parametrized◊ principles.
\newblock {\em The Journal of Symbolic Logic}, 73(3):752--764, 2008.

\bibitem{spres}
Tadatoshi Miyamoto.
\newblock $\omg$-{S}ouslin trees under countable support iterations.
\newblock {\em Fundamenta Mathematicae}, 142(3):257--261, 1993.

\bibitem{justinmin}
J.~T. Moore.
\newblock $\omega_1$ and $-\omega_1$ may be the only minimal uncountable linear
  orders.
\newblock {\em Michigan Math. J}, 55(2):437--457, 2007.

\bibitem{pardiamond}
Justin Moore, Michael Hru{\v{s}}{\'a}k, and Mirna D{\v{z}}amonja.
\newblock Parametrized $\diamondsuit$ principles.
\newblock {\em Transactions of the American Mathematical Society},
  356(6):2281--2306, 2004.

\bibitem{DHM}
Justin Moore, Michael Hru{\v{s}}{\'a}k, and Mirna D{\v{z}}amonja.
\newblock Parametrized♢ principles.
\newblock {\em Transactions of the American Mathematical Society},
  356(6):2281--2306, 2004.

\bibitem{whitehead1}
S.~Shelah.
\newblock Whitehead groups may be not free, even assuming {CH}, {I}.
\newblock {\em Israel Journal of Mathematics}, 28(3):193--204, 1977.

\bibitem{whitehead2}
S.~Shelah.
\newblock Whitehead groups may not be free even assuming {CH}, {II}.
\newblock {\em Israel Journal of Mathematics}, 35(4):257--285, 1980.

\bibitem{shelah1989consistent}
S~Shelah.
\newblock A consistent counterexample in the theory of collectionwise
  {H}ausdorff spaces.
\newblock {\em Israel Journal of Mathematics}, 65(2):219--224, 1989.

\bibitem{shelah1974infinite}
Saharon Shelah.
\newblock Infinite abelian groups, {W}hitehead problem and some constructions.
\newblock {\em Israel Journal of Mathematics}, 18(3):243--256, 1974.

\bibitem{shelah2017proper}
Saharon Shelah.
\newblock {\em Proper and improper forcing}, volume~5.
\newblock Cambridge University Press, 2017.

\bibitem{soukuptrees}
D{\'a}niel~T. Soukup.
\newblock Trees, ladders and graphs.
\newblock {\em Journal of Combinatorial Theory, Series B}, 115:96--116, 2015.

\bibitem{treeunifA}
Daniel~T. Soukup.
\newblock Ladder system uniformization on trees {I}: colouring ladders.
\newblock {\em preprint, arXiv: https://arxiv.org/abs/1806.03867}, 2018.

\bibitem{treeunifB}
Daniel~T. Soukup.
\newblock Ladder system uniformization on trees {II}: growing trees.
\newblock {\em preprint, arXiv: https://arxiv.org/abs/1806.03867}, 2018.

\bibitem{soukup2018model}
D{\'a}niel~T Soukup.
\newblock A model with {S}uslin trees but no minimal uncountable linear orders
  other than $\omega_1 $ and $-\omega_1$.
\newblock {\em Israel Journal of Mathematics (to appear), arXiv:1803.03583},
  2018.

\bibitem{stronglysurj}
D{\'a}niel~T. Soukup.
\newblock Uncountable strongly surjective linear orders.
\newblock {\em Order. https://doi.org/10.1007/s11083-018-9454-7}, 2018.

\bibitem{treeladder}
Z.~Spasojevi{\'c}.
\newblock Ladder systems on trees.
\newblock {\em Proceedings of the American Mathematical Society},
  130(1):193--203, 2002.

\bibitem{MR2979581}
Franklin~D. Tall.
\newblock {${\rm PFA}(S)[S]$}: more mutually consistent topological
  consequences of {PFA} and {$V=L$}.
\newblock {\em Canad. J. Math.}, 64(5):1182--1200, 2012.

\bibitem{stevocont}
S.~Todorcevic.
\newblock Stationary sets, trees and continuums.
\newblock {\em Publ. Inst. Math. (Beograd) (N.S.)}, 29(43):249--262, 1981.

\bibitem{stevotrees}
S.~Todorcevic.
\newblock Trees and linearly ordered sets.
\newblock {\em Handbook of set-theoretic topology}, pages 235--293, 1984.

\bibitem{MR631563}
Stevo~B. Todorcevic.
\newblock Trees, subtrees and order types.
\newblock {\em Ann. Math. Logic}, 20(3):233--268, 1981.

\bibitem{watson1986locally}
Stephen Watson.
\newblock Locally compact normal meta-{L}indel{\"o}f spaces may not be
  paracompact: an application of uniformization and {S}uslin lines.
\newblock {\em Proceedings of the American Mathematical Society},
  98(4):676--680, 1986.

\bibitem{yorioka2010uniformizing}
Teruyuki Yorioka.
\newblock Uniformizing ladder system colorings and the rectangle refining
  property.
\newblock {\em Proceedings of the American Mathematical Society},
  138(8):2961--2971, 2010.

\end{thebibliography}

\end{document}